\documentclass[11pt]{article}
\usepackage{amsthm}
\usepackage{amscd}
\usepackage{latexsym}
\usepackage{graphicx}
\usepackage{amsmath}
\usepackage{amsfonts}
\usepackage{float}
\usepackage{amssymb}
\usepackage{enumerate}
\usepackage{multirow}
\usepackage{floatflt}
\usepackage{color}

\newtheorem{theorem}{Theorem}
\newtheorem{lemma}{Lemma}

\newtheorem{definition}{Definition}
\newtheorem{proposition}{Proposition}

\theoremstyle{definition}
\newtheorem{example}{Example}
\theoremstyle{remark}
\newtheorem{remark}{Remark}

\begin{document}

\begin{center} \huge\bf
Enlarging a connected graph  while keeping
entropy and spectral radius: self-similarity techniques
\end{center}

\bigskip
\medskip
\begin{center}
\textsc{Alberto Seeger}\footnote{University of  Avignon, Department of Mathematics, 33 rue Louis Pasteur,
84000 Avignon, France (e-mail: aseegerfrance@gmail.com).} \quad and \quad\textsc{David Sossa}\footnote{Universidad de O'Higgins, Instituto de Ciencias de la Ingenier\'ia, Av.\,Libertador Bernardo O'Higgins 611, Rancagua, Chile (e-mail: david.sossa@uoh.cl).   }
\end{center}
\bigskip

\begin{quote}
 {\small \textbf{Abstract}.
 This work   is about self-similar sequences of growing connected graphs. We   explain how to construct such sequences and why they are important.  We show for instance that all the connected graphs in a self-similar sequence have  not only the same entropy, but also the same spectral radius.

\bigskip

{\it Mathematics Subject Classification}: 05C50, 15A42.\\
{\it Key words}:  connected graph, spectral radius,  graph entropy,   automorphism similarity,  orbit partition, orbital similarity.
}\end{quote}

\bigskip

\section{Introduction}\label{se:intro}

Self-similarity is an expression used in topology, statistics, network theory, fluid dynamics, and other areas of mathematics and physics. Such an expression refers to the phenomenon where a certain property of an object  is preserved with respect to scaling.
Informally speaking,  a self-similar object is exactly or approximately similar to a part of itself.
In  graph theory, the idea of self-similarity can be formalized in many ways. Our  definition of graph self-similarity is based  on the concept of  orbital similarity  between graphs.  Of course, the graphs under comparison are not necessarily of the same   order.  In fact, we are specially interested in describing a situation  in which  a small connected graph  looks somehow similar to a big connected graph. The small one is neither an induced subgraph nor a proper subgraph of the big one, but in a way it is a part of it.

For a  smooth introduction  into the topic, consider as toy example a sequence $\{G_k: k\geq 1\}$
of growing connected graphs as those shown in Figure\,\ref{Fig:molecular}.
As we move from left to right,  the successive graphs  in Figure\,\ref{Fig:molecular} get larger in order (number of vertices) and size (number of edges).   Despite a  change in scale, we observe  a certain resemblance  among  all these graphs. If we pick any two of them, say $G_2$ and $G_4$, then a quick computation shows that both graphs not only have the same spectral radius, but they have also plenty  of other  graph invariants in common: maximal degree,  minimal degree,  edge-vertex ratio, and so  on.
    The caption of Figure\,\ref{Fig:molecular} needs of an explanation.   The purpose of this work is precisely to  explain why orbitally similar graphs have the same entropy and the same spectral radius. We  explain also  how to construct such sequences of growing connected graphs. As just mentioned, there are plenty of properties shared by  the $G_k$'s in Figure\,\ref{Fig:molecular}.   Some of these common properties are a consequence of the very definition  of orbital similarity that is  proposed.   One should accept however  some unavoidable divergences between the $G_k$'s due to a change in  scale.
It is hopeless trying to preserve  for instance   the  standard density index, cf.\,Proposition\,\ref{pr:density},  and  other  graph invariants.
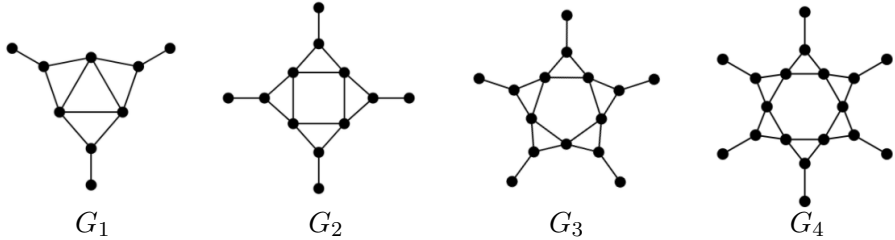
\begin{figure}[!ht]
     \centering
    \includegraphics[width=0.75\textwidth]{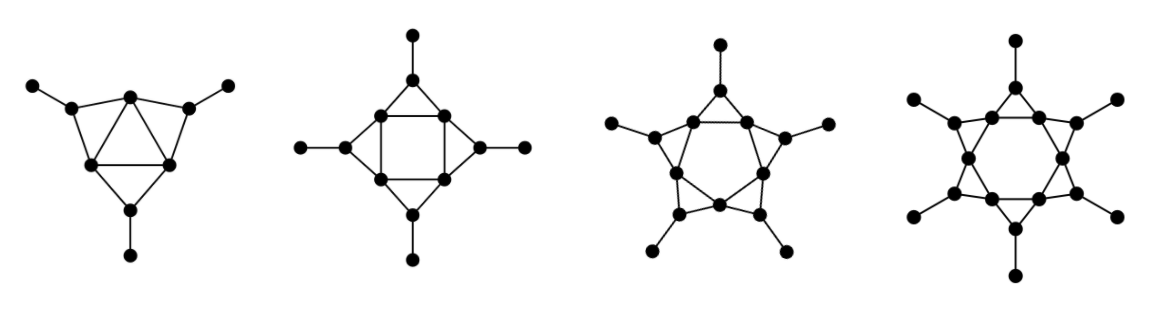}\put(-112,0){$G_1$}\put(-81,0){$G_2$}\put(-49,0){$G_3$}\put(-17,0){$G_4$}
    \caption{\small   Orbitally similar graphs. In particular, they have  equal entropy and equal spectral radius. }\label{Fig:molecular}
    \end{figure}

    The organization of the paper is as follows. Section\,\ref{se:background}  settles  the  background and introduces the concept of  orbital similarity between graphs. With the help of such a concept, we explain what  means a self-similar sequence of connected graphs emanating from a given connected graph called  the  seed. The seed plays the role of a prototype graph that we wish to reproduce  or replicate {\it ad infinitum}.  Section\,\ref{se:os} is devoted to the analysis of orbital similarity  as a concept of interest by itself.  Section\,\ref{se:construct} proposes  several practical mechanisms for constructing self-similar sequences.
 Section\,\ref{se:preserve} gives a list of graph invariants  that are always  preserved in a self-similar sequence and  it mentions also some  graph invariants that are preserved only under additional assumptions on the seed.

\section{Setting the background}\label{se:background}
All graphs are assumed  to be  finite,  undirected, without loops and multiple edges.
 A basic ingredient of the discussion is the classical concept of automorphism similarity between vertices. As in   Harary and  Palmer\,\cite{HP},  two vertices   $u$ and $v$ of a graph $G$  are  {\it automorphically similar}  if $\varphi(u)=v$ for some  $\varphi\in {\rm Aut}(G)$, where  ${\rm Aut}(G)$ is the group of automorphisms of $G$.
Several equivalent characterizations of automorphism similarity are proposed in Seeger and Sossa\,\cite[Theorem\,2.1]{SeSo1}.
Automorphism similarity is an equivalence relation on the vertex set  of a graph. The automorphism similarity classes are called {\it orbits }  of the graph.  In other words, an orbit of a  graph is the set of all vertices automorphically similar to a given vertex.  By abuse of language, the set
\begin{equation}\label{partition}
{\rm Orb}(G):=\{\mathcal O_1,\ldots,  \mathcal O_\ell\}
 \end{equation}
of orbits is  called the  {\it orbit partition} \color{black} of $G$. It is of course the  vertex set
$V_G:=\{v_1, \ldots, v_n\}$
 that  is being partitioned  and not the graph itself.  The degree of a vertex $v\in V_G$ is denoted by $d_G(v)$. Vertices on a same orbit  not only have the same degree,   but also other properties in common (for instance, they have the  same neighborhood degree sequence, cf.\,\cite[Theorem\,3.1]{SeSo1}).
The  number $\ell$ of orbits is an interesting structural parameter of a graph.  If $\ell$ is small
compared to the number $n$ of vertices, then the graph can be  viewed as highly symmetric. One extreme case of
symmetry is a graph with only one orbit. Such sort of  graph is called {\it vertex-transitive}.  The other extreme is a graph with as many orbits as vertices. In the  latter case, the graph is called {\it asymmetric}.    The first graph  in Figure\,\ref{Fig:partition} is vertex-transitive, the second one  is asymmetric, and the last one corresponds to an intermediate situation: the number of orbits is greater than one, but smaller than the number of vertices.
\begin{figure}[!ht]
     \centering
    \includegraphics[width=0.6\textwidth]{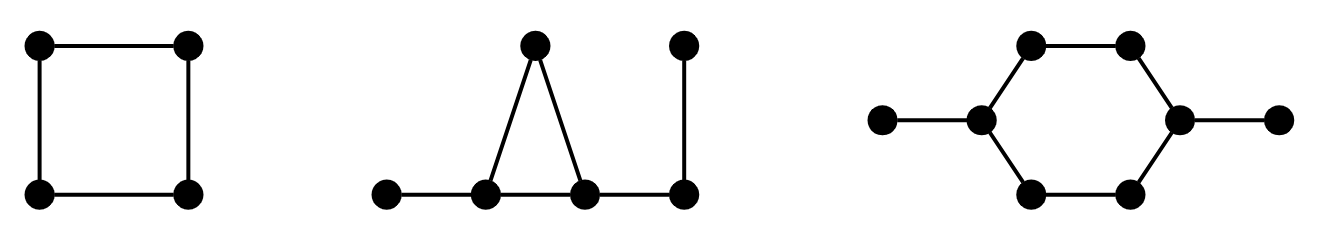}
    \put(-98,-1){\scriptsize$\alpha_1$}\put(-98,17){\scriptsize$\alpha_1$}\put(-86,17){\scriptsize$\alpha_1$}\put(-86,-1){\scriptsize$\alpha_1$}\put(-60,17){\scriptsize$\alpha_6$}\put(-49,17){\scriptsize$\alpha_1$}\put(-49,-1){\scriptsize$\alpha_2$}\put(-57,-1){\scriptsize$\alpha_3$}\put(-65,-1){\scriptsize$\alpha_4$}\put(-73,-1){\scriptsize$\alpha_5$}\put(-16,17){\scriptsize$\alpha_1$}\put(-16,-1){\scriptsize$\alpha_1$}\put(-25,-1){\scriptsize$\alpha_1$}\put(-25,17){\scriptsize$\alpha_1$}\put(-12,5){\scriptsize$\alpha_2$}\put(-29,5){\scriptsize$\alpha_2$}\put(-5,5){\scriptsize$\alpha_3$}\put(-36,5){\scriptsize$\alpha_3$}
\caption{\small   Vertex-transitivity,  asymmetry, and   partial symmetry. }\label{Fig:partition}
    \end{figure}
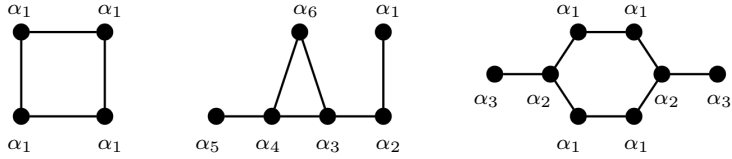

It is  helpful to view   the Greek letters $\alpha_1,\ldots, \alpha_\ell$  as being  colors, for instance
    \begin{equation}\label{colors}
 (\alpha_1,\alpha_2,\alpha_3,\ldots, \alpha_\ell) \,=\, (\mbox{blue}, \mbox{green}, \mbox{yellow}, \ldots, \mbox{red}).
  \end{equation}
  In such a case, a vertex  indicated   with  the letter $\alpha_1$ is a blue  vertex, and so on.  The number $\ell$ of orbits corresponds to the numbers of colors used to paint the vertices. The smaller the number of colors, the higher the degree of symmetry of the  graph.   If the  graph $G$  is connected, then the $\alpha_i$'s can be viewed as components of the  principal eigenvector (or Perron eigenvector)   of $G$.
We  briefly explain the link between the orbit partition and the principal eigenvector of a connected graph.
  Let $A_G$ be the adjacency matrix of $G$. This is a matrix of order $n=\vert G\vert$. It is sometimes convenient to index the entries of such a matrix  with the vertices of $G$, in which case  $A_G:V_G\times V_G\to \mathbb{R}$ is a function  given by
  $ A_G(u,v)= 1$ if $\{u,v\}$ is an edge,  and  $ A_G(u,v)= 0$ otherwise.
  Let $\varrho (G)$ be the spectral  radius of $G$.
  The  {\it principal eigenvector} of a connected graph $G$ is   denoted by $x_G$ and it is defined as the unique  function $x:V_G\to \mathbb{R}$ satisfying
  \begin{equation}\label{functional}
    \sum_{v\in V_G}A_G(u,v) x(v)= \varrho(G)\, x(u)
  \end{equation}
for all $u\in V_G$,  together with  the normalization  condition
   $\sum_{v\in V_G}x(v)= 1$.
  The  existence and uniqueness of $x_G$ is guaranteed by the Perron-Frobenius theorem.
  Instead of viewing $x_G$ as a function on $V_G$, we can equally well see $x_G$ as an $n$-dimensional column vector.  For this, it suffices to write the functional equation  (\ref{functional}) in the matrix format
   $ A_G x =\varrho(G)\, x. $
The normalization condition simply says that the components of $x\in \mathbb{R}^n$ sum up to one. Such components are positive of course.  The following proposition is part of the folklore in spectral graph theory.

  \begin{proposition} \label{pr:perron}  Let $u$ and $v$ be automorphically similar vertices of a connected graph $G$. Then $x_G(u)= x_G(v)$.
  \end{proposition}

Proposition\,\ref{pr:perron} asserts that $x_G$ is constant on each orbit of $G$.  This result is specially useful when $G$ has only  few orbits, because in such a case there are only few constants to be determined for knowing the entire principal eigenvector.
The reverse of Proposition\,\ref{pr:perron} is not true, i.e., the equality  $x_G(u)= x_G(v)$ does not imply that $u$ and $v$ are automorphically similar. Think for instance of a regular connected graph that is not vertex-transitive.

\subsection{Orbital similarity between graphs}\label{ss:orbital-similarity}

We now explain what does it mean that two graphs $G$ and $H$ of possibly different order are orbitally similar. For better understanding our motivation, think of $G$ as a connected graph with  a few vertices and $H$ as a connected graph with a large number of vertices.
Two connected graphs can be orbitally similar despite  being of a completely different scale. Orbital similarity is not a matter of order or size, but a matter of orbit structure.  To start with, orbital similarity requires both graphs to have the same number of orbits, say
\begin{equation*}
{\rm Orb}(G)=\{\mathcal O_1,\ldots,  \mathcal O_\ell\}\,,\;\; {\rm Orb}(H)=\{\mathcal Q_1,\ldots,  \mathcal Q_\ell\}.
 \end{equation*}
But much more is needed.  Two graphs with the same number of orbits may look quite different after all.  Figure\,\ref{Fig:3orbits}  displays a sample $\{Z_1,\ldots, Z_8\}$  of connected graphs  with exactly $3$ orbits each.  A few pairs look similar, but most pairs look dissimilar.
\begin{figure}[!ht]
     \centering
       \includegraphics[width=0.7\textwidth]{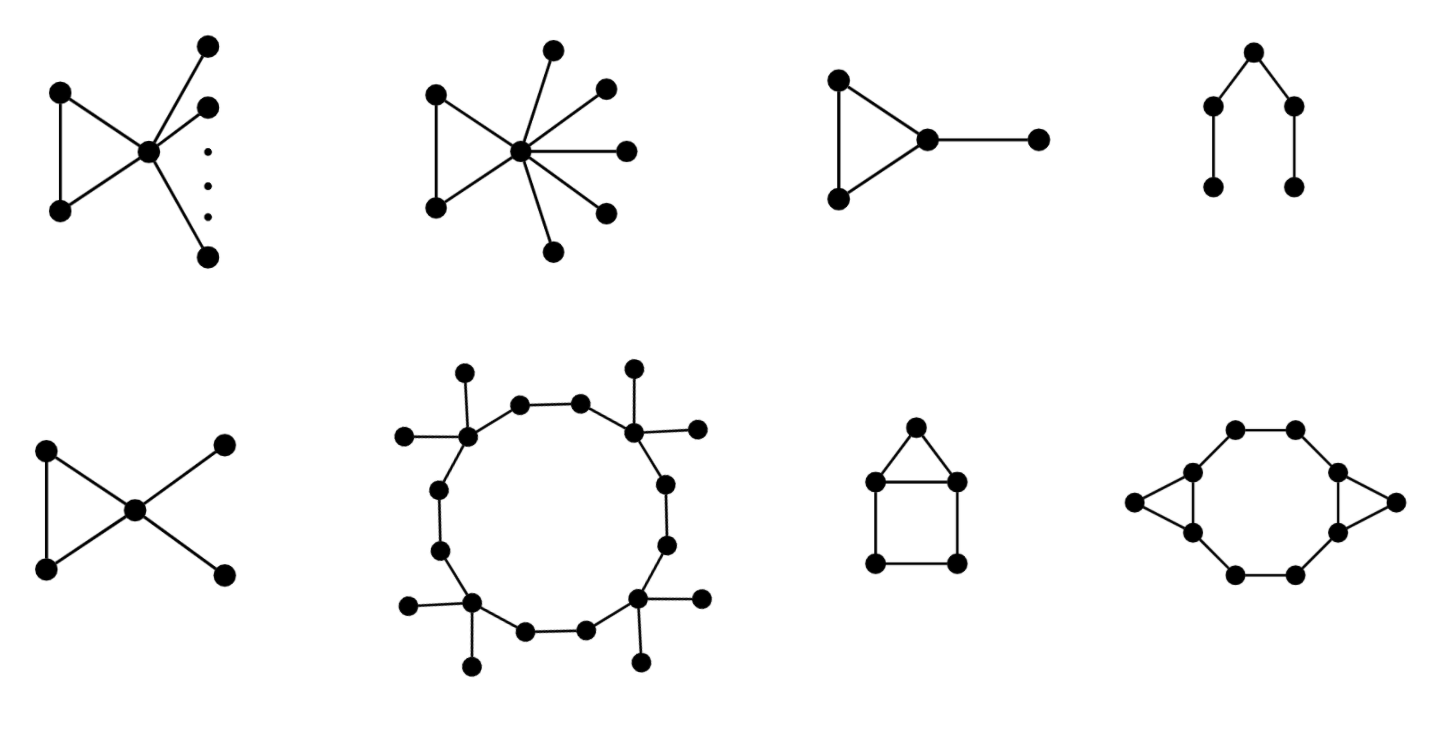}\put(-96,55){\scriptsize$1$}\put(-96,50){\scriptsize$2$}\put(-97,38){\scriptsize$17$}\put(-105,33){$Z_1$}\put(-75,33){$Z_2$}\put(-43,33){$Z_3$}\put(-17,33){$Z_4$}
       \put(-105,2){$Z_5$}\put(-75,2){$Z_6$}\put(-43,2){$Z_7$}\put(-17,2){$Z_8$}
    \caption{\small  A sample of  connected graphs with $3$ orbits each. }\label{Fig:3orbits}
    \end{figure}
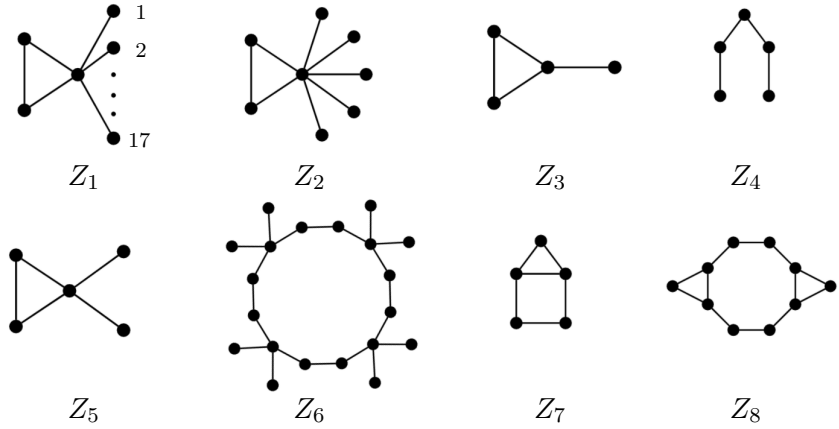
In addition to the  equality $\vert {\rm Orb}(G)\vert = \vert {\rm Orb}(H)\vert$,  we should impose  certain relations between the  orbits of $G$ and  $H$.
Let the  {\it orbit distribution vector} \color{black}  of $G$ be defined as
 \begin{equation}\label{odv}
\omega(G):= \left (\frac{\vert \mathcal O_1\vert}{\vert G\vert},\ldots , \frac{\vert \mathcal O_\ell\vert}{\vert G\vert}\right)^\downarrow,
\end{equation}
 where  $\xi^\downarrow$ denotes the nonincreasing rearrangement  of  $\xi\in \mathbb{R}^\ell$. Of course, the down-arrow superscript  in (\ref{odv}) can be dropped  if  the orbits  of $G$ are labeled so that
\begin{equation}\label{cardi}
 \vert \mathcal O_1\vert \geq \ldots \geq \vert  \mathcal O_\ell\vert\,.
 \end{equation}
 The components of $\omega(G)$ are positive and sum up to one.  Such a probability   vector
determines the entropy  of $G$, which is a graph invariant given by
 \begin{equation}\label{entropy}
 {\rm Ent}(G):=\Phi(\omega(G))\,
 \end{equation}
 with $\Phi(\xi):=-\sum_{i=1}^\ell\xi_i\log \xi_i$. The  logarithm is base $2$, but any another base is also acceptable.
 See Mowshowitz\,\cite{Mow} and Mowshowitz and Mitsou\,\cite{MM} for a discussion on  the expression (\ref{entropy}) and the survey paper of Dehmer and Mowshowitz\,\cite{DM} for general information on graph entropy measures.   The orbit distribution vector and the entropy of $H$ are defined analogously.
 If $G$ and $H$ were to be  orbitally similar, then it makes sense to  request
 \begin{equation}\label{request}
 \omega(G)=  \omega(H).
 \end{equation}
 If the orbits of $G$ and $H$ are labeled as in (\ref{cardi}), then equality (\ref{request})
is a short way of saying that each  $\mathcal O_i$   has the same relative size as the corresponding $\mathcal Q_i$.  Note that we are not comparing  $\vert \mathcal O_i\vert$ and  $\vert \mathcal Q_i\vert$, but the ratios $\vert \mathcal O_i\vert/\vert G\vert$ and $\vert \mathcal Q_i\vert/\vert H\vert$.
Let us have a closer look at the graphs of  Figure\,\ref{Fig:3orbits}.  As shown in  Table\,\ref{Tab:entropy}, such graphs are arranged in fact by nondecreasing level  of entropy.
 \begin{table}[htbp]
\begin{center}
\begin{small}
\scriptsize
\caption{\small   Orbit distribution vector and entropy of the graphs in Figure\,\ref{Fig:3orbits}.  }\label{Tab:entropy}
\bigskip
\begin{tabular}{|c| c | c | c | c | c | c | c | c |c | c | c | c |c | c | c | c |c | c | c | c |c | c | c | c |c | c | c | c |}
\hline
G          &$\vert G\vert$ &  $\omega(G)$         & ${\rm Ent}(G)$     \\
\hline
$Z_1$      &  20           & $(17/20, 2/20, 1/20)$   &  0.7476           \\
$Z_2$      &   8           & $(5/8, 2/8, 1/8)$       &  1.2988           \\
$Z_3$      &   4           & $(2/4, 1/4, 1/4)$       &  1.5000            \\
$Z_4$      &   5           & $(2/5, 2/5, 1/5)$       &  1.5219            \\
$Z_5$      &   5           & $(2/5, 2/5, 1/5)$       &  1.5219           \\
$Z_6$      &  20           & $(8/20, 8/20, 4/20)$    &  1.5219           \\
$Z_7$      &   5           & $(2/5, 2/5, 1/5)$       &  1.5219           \\
$Z_8$      &  10           & $(4/10, 4/10, 2/10)$    &  1.5219          \\
\hline
\end{tabular}
\end{small}
\end{center}
\end{table}
 The graph $Z_1$ has a much lower entropy than all others. Without hesitation,  we declare $Z_1$  dissimilar from the others.  Since $Z_2$ has  its  own  entropy level,  this graph is also considered as dissimilar from the others.  The same remark applies to $Z_3$, even if   ${\rm Ent}(Z_3)$ is not too far  from ${\rm Ent}(Z_4)$.  Concerning the graphs from $Z_4$ to $Z_8$ inclusive,  we are momentarily in a troublesome situation. These  $5$ graphs have not only the same entropy, but also the same orbit distribution vector.  We feel  however that
 $Z_4$ is dissimilar from $Z_5$.  Having the same orbit distribution vector does not seem to be strong enough to recover an intuitive notion of similarity.  For this reason,  we give little  credit to the orbit distribution vector and put our faith in the  orbit divisor matrix. This point is explained   next. As it  is well known,  orbit partitions  are  equitable partitions, cf.\,\cite[Proposition\,3.2]{BFW}.  That (\ref{partition}) is an equitable partition means that, for all $i,j\in \{1,\ldots, \ell\}$, the number
\begin{equation}\label{defSG}  s_{i,j}\,=\,\sum_{v\in \mathcal O_j}A_G(u,v)
\end{equation}
of vertices in $\mathcal O_j$ that are adjacent to $u\in \mathcal O_i$  does not depend on the choice of the  vertex $u$ in $\mathcal O_i$. Each  vertex of color $\alpha_i$ is adjacent to exactly $ s_{i,j}$  vertices of color  $\alpha_j$.  For instance,  if we use the color convention (\ref{colors}), then each  blue vertex is adjacent to exactly
 $s_{1,1}$ blue vertices,
 $s_{1,2}$  green  vertices,
 $s_{1,3}$ yellow vertices,
   and so on. Beware that  $s_{i,j}$ is not necessarily equal to $s_{j,i}$.
The possibly non-symmetric  matrix
 $ S_G:=[s_{i,j}]$
  reflects  the orbit structure of the graph $G$.  In the terminology of  Barrett et al.\,\cite[Definition\,2.1]{BFW}, $S_G$ corresponds to the divisor matrix (or quotient matrix) associated to the orbit partition of $G$.   For the sake of brevity, we call $S_G$ the {\it orbit divisor matrix}  of $G$.
 \begin{example} \label{ex:same}
 The orbits of the graphs $G$ and $H$ shown in Figure\,\ref{Fig:relabeling} are
  \begin{eqnarray*}
   &&\mathcal O_1= \{v_1, v_5\},\hskip 1,5cm\mathcal O_2= \{v_3, v_4\},\hskip 1,4cm \mathcal O_3= \{ v_2\}, \\
    && \mathcal Q_1= \{u_1, \tilde u_1, u_5, \tilde u_5\},\; \;\, \mathcal Q_2= \{u_3, \tilde u_3, u_4, \tilde u_4\}, \;\; \mathcal Q_3= \{ u_2,\tilde u_2\}\,,
     \end{eqnarray*}
    respectively. For each $i\in \{1,2,3\}$,   $\mathcal Q_i$ has twice as many elements as  $\mathcal O_i$. We get
$$\omega(G)= (2/5, 2/5, 1/5)=(4/10, 4/10, 2/10)= \omega(H).$$ In this example,  $G$ and $H$ have not only the same orbit distribution vector, but also the same orbit divisor  matrix, namely,
\begin{equation}\label{SGSH}
 S_G= S_H = \left[
                \begin{array}{ccc}
                  1 & 0 & 1 \\
                    0 & 0 & 1 \\
                      2 & 2 & 0 \\
                \end{array}
              \right].
\end{equation}
\end{example}
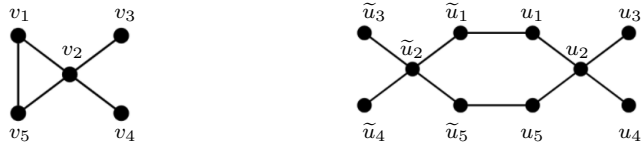
\begin{figure}[!ht]
     \centering
   \includegraphics[width=0.55\textwidth]{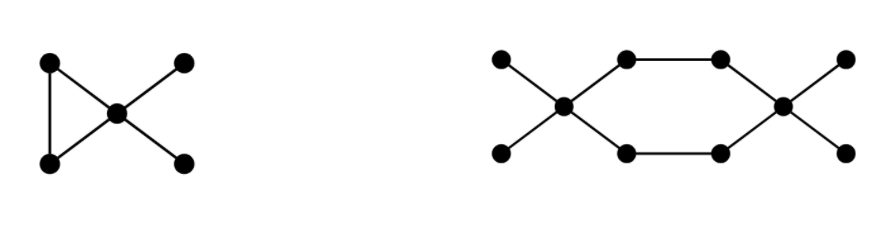}\put(-87,19){\scriptsize$v_1$}\put(-80,14){\scriptsize$v_2$}\put(-73,19){\scriptsize$v_3$}\put(-87,3){\scriptsize$v_5$}\put(-73,3){\scriptsize$v_4$}\put(-40,19){\scriptsize$\widetilde{u}_3$}\put(-35,14){\scriptsize$\widetilde{u}_2$}\put(-29,19){\scriptsize$\widetilde{u}_1$}\put(-29,3){\scriptsize$\widetilde{u}_5$}\put(-40,3){\scriptsize$\widetilde{u}_4$}\put(-19,19){\scriptsize$u_1$}\put(-19,3){\scriptsize$u_5$}\put(-13,14){\scriptsize$u_2$}\put(-6,19){\scriptsize$u_3$}\put(-6,3){\scriptsize$u_4$}
    \caption{\small   Both graphs have the same  orbit divisor matrix.  }\label{Fig:relabeling}
    \end{figure}
 \begin{example}\label{ex:path}
    Consider  the path  $P_5$ with  vertices $\{v_1,\ldots,v_5\}$ arranged in the  usual way. We have
$\mathcal O_1= \{v_1, v_5\}$, $\mathcal O_2= \{v_2, v_4\}$, and $\mathcal O_3= \{ v_3\}$.
The  orbit distribution vector
 $\omega(P_5)= (2/5, 2/5, 1/5)$ is as in Example\,\ref{ex:same}, but the orbit divisor matrix
\begin{equation}\label{Spath} S_{P_5}= \left[
                \begin{array}{ccc}
                  0 & 1 & 0 \\
                    1 & 0 & 1 \\
                    0 & 2 & 0 \\
                \end{array}
              \right]
 \end{equation}
 is different from (\ref{SGSH}).
\end{example}
 \vskip 0,1cm

 The next definition is restricted to connected graphs, but, in principle,  disconnected graphs could also be accommodated.

 \begin{definition} \label{de:os}
 Two  connected graphs $G$ and $H$ are  orbitally similar if
 $S_G=S_H.$
 \end{definition}
 \vskip 0,1cm
  A technical remark concerning the equality  $S_G=S_H$ is in order.
 Suppose that the orbits $\mathcal O_1,\ldots, \mathcal O_\ell$ of  the graph $G$ are labeled as in  (\ref{cardi}), i.e., nonincreasingly with respect to cardinality. If the sizes of the orbits are all different, then  there is no ambiguity with such a  labeling strategy.  However, if there are at least two orbits of equal size, then the $\ell$-tuple
$ (\mathcal O_1,\ldots, \mathcal O_\ell )$ is not defined in a unique way. For instance, in  Example\,\ref{ex:path}, we could have  chosen
\begin{eqnarray*}
\mathcal O_1^\prime &=& \mathcal O_{\pi(1)}\;=\; \mathcal O_{2}\;=\;\{v_2, v_4\}, \\
\mathcal O_2^\prime&=&  \mathcal O_{\pi(2)}\;=\; \mathcal O_{1}\;=\;\{v_1, v_5\}, \\
\mathcal O_3^\prime &=& \mathcal O_{\pi(3)}\;=\; \mathcal O_{3}\;=\; \{ v_3\}.
\end{eqnarray*}
The corresponding orbit divisor matrix
\begin{equation*}  \left[
                \begin{array}{ccc}
                  0 & 1 & 1 \\
                    1 & 0 & 0 \\
                    2 & 0 & 0 \\
                \end{array}
              \right]
 \end{equation*}
 is  the same as (\ref{Spath}), but only  up to permutation similarity transformation.
 This observation is formalized in the next easy proposition.
 \begin{proposition}\label{pr:permutation}
 Let $G$ be a connected graph with orbits  $\mathcal O_1,\ldots,  \mathcal O_\ell$ labeled in any order. Let $S_G$ and $S_G^\pi$ be the orbit divisor matrices relative to   $ (\mathcal O_1,\ldots, \mathcal O_\ell )$ and  $ (\mathcal O_{\pi(1)},\ldots, \mathcal O_{\pi(\ell)}) $, respectively, where  $\pi$ is  any permutation of $\ell$ elements.
Then $S_G^\pi =P^\top S_G P$, where $P$ is the permutation matrix associated to $\pi$.
 \end{proposition}
 \begin{proof}  The $k$-th column of $P$ is $e_{\pi(k)}$,
 where $\{e_1,\ldots, e_\ell\}$ is the canonical basis  of $\mathbb R^\ell$. Hence,
 \begin{equation*}
 (P^\top S_G P)_{ij}\,=\,(Pe_i)^\top S_G(Pe_j)\,=\,
e_{\pi(i)}^\top S_G e_{\pi(j)}\,=\,s_{\pi(i),\pi(j)}\,=\,(S^\pi_G)_{i,j}
 \end{equation*}
 for all $i,j\in\{1,\ldots,\ell\}$.
\end{proof}

The matrices $S_G$ and $S_G^\pi$ in Proposition\,\ref{pr:permutation} are equal up to  permutation similarity transformation.
The equality  $S_G=S_H$ in Definition\,\ref{de:os} is to be understood relative to suitable choices of
 $ (\mathcal O_1,\ldots, \mathcal O_\ell )$ and $ (\mathcal Q_1,\ldots, \mathcal Q_\ell )$. Once we fix the order of the $\mathcal O_i$'s, we may need to relabel the $\mathcal Q_i$'s for ensuring that $S_G$ and $S_H$  are truly equal and not just equal  up to  permutation similarity transformation. We now are ready to introduce the concept of self-similar sequence.

\begin{definition}\label{de:self} Let $G$ be a connected graph. A  self-similar sequence emanating from $G$ is  a sequence
$\mathbb{G}=\{G_k: k\geq 1\} $
of connected graphs such that:
\begin{eqnarray}\label{c1}
&&\vert G_{1}\vert < \vert G_{2}\vert <\vert G_{3}\vert <\ldots \,,\\ \label{c2}
&&G_1, G_2, G_3,\ldots, \mbox{are pairwisely orbitally similar},\\ \label{c3}
&&G_1  \mbox{ is isomorphic to }G.
\end{eqnarray}
\end{definition}
\vskip 0,2cm

 The graph $G$ in Definition\,\ref{de:self} is called   the {\it seed}  of the sequence $\mathbb{G}$. Up to isomorphism, $G$  is simply the first term  of the sequence.  Condition  (\ref{c2}) can be  reformulated by saying that each $G_k$  is orbitally similar to $G$. Note that we could equally well express (\ref{c2}) by saying that each $G_k$  is orbitally similar to $G_1$.
Sometimes the seed of a self-similar sequence is irrelevant in the discussion: what truly  matters about the sequence $\mathbb{G}$   is the fact of being  self-similar and not from where it emanates.    In such a case, we simply omit mentioning the seed and forget  condition  (\ref{c3}).  Condition (\ref{c1})  says that the $G_k$'s are getting bigger in order. Such a growth condition serves to avoid repetitions among the $G_k$'s and  to ensure that $\mathbb{G}$ is an infinite set.

\section{Analysis of orbital similarity}\label{se:os}
In order to construct and manipulate self-similar sequences, it is necessary to have a good understanding of the concept of orbital similarity between connected graphs. This section is devoted to the analysis of such a  concept.

 \begin{theorem}\label{th:implications} Let $G$ and $H$ be connected graphs. Then each of the following conditions
implies but it is not implied by the next one:
 \begin{itemize}
 \item [(a)] $G$ and $H$ are isomorphic.
 \item [(b)] $G$ and $H$ are orbitally similar and  have the same order.
 \item [(c)] $G$ and $H$ are orbitally similar.
 \item [(d)] $G$ and $H$ have the same orbit distribution vector.
 \item [(e)] $G$ and $H$ have the same entropy.
 \end{itemize}
 \end{theorem}
 \begin{proof}
 That (a) implies (b) is clear, but the reverse implication is not true, cf.\,Figure\,\ref{Fig:domino}.  That (b) implies (c) is tautological, but orbitally similar connected graphs are not necessarily  of the same order,  cf.\,Figure\,\ref{Fig:molecular}.
 Next, we prove that   (c) implies  (d). Suppose  that $G$ and $H$ are orbitally similar. We  label the orbits $\mathcal O_1,\ldots,  \mathcal O_\ell$  of $G$ as in (\ref{cardi}) and, afterwards, we label the orbits  $\mathcal Q_1,\ldots,  \mathcal Q_\ell$ of $H$  so as to have
$S_G=S_H$.  The common orbit divisor matrix is denoted by $S$.  We claim that
 \begin{equation}\label{claima}
\left(\frac{\vert \mathcal O_1\vert}{|G|},\ldots, \frac{\vert \mathcal O_\ell\vert}{|G|}\right) \;=\;\left(\frac{\vert \mathcal Q_1\vert}{|H|},\ldots, \frac{\vert \mathcal Q_\ell\vert}{|H|}\right)\,.
\end{equation}
The components of the first vector are arranged in nonincreasing order. If the claim (\ref{claima}) is true, then  the components of the second vector are also arranged in nonincreasing order and we obtain $\omega(G)=\omega(H)$ as  desired.
 Let  $s_{i,j}$ be the $(i,j)$-entry of $S$. Since $S=S_G$, we see that $|\mathcal O_i|s_{i,j}$ is equal to the number of edges joining $\mathcal O_i$ with $\mathcal O_j$,   which in turn  is equal to $|\mathcal O_j|s_{j,i}$.   In short,
    \begin{equation}\label{e1}
 |\mathcal O_i|s_{i,j}=|\mathcal O_j|s_{j,i}.
 \end{equation}
  Since $S=S_H$, the same argument yields
  \begin{equation}\label{e2}
  |\mathcal Q_i|s_{i,j}=|\mathcal Q_j|s_{j,i}.
 \end{equation}
 The next step consists in proving that
 \begin{equation} \label{nextstep}
 \frac{|\mathcal O_p|}{|\mathcal O_q|}=\frac{|\mathcal Q_p|}{|\mathcal Q_q|}
 \end{equation}
 for all $p,q\in\{1,\ldots,\ell\}$.  If $p=q$, then both sides of (\ref{nextstep}) are equal to $1$ and we are done. Suppose that $p\not =q$. If $s_{p,q}$ is nonzero, then
 the choice $(i,j)=(p,q)$ in (\ref{e1})-(\ref{e2}) yields
 $$  \frac{|\mathcal O_p|}{|\mathcal O_q|}\,=\,\frac{ s_{q,p}}{s_{p,q}}\,=\, \frac{|\mathcal Q_p|}{|\mathcal Q_q|}$$
 and we are done again.
 Suppose that $s_{p,q}=0$. In such a case, also $s_{q,p}=0$ and  there is no edge connecting  the orbits $\mathcal O_p$ and $\mathcal O_q$. More precisely, no vertex from one orbit is adjacent to a vertex from the other orbit.
   Given that we are dealing with connected graphs, there are indices  $r_0,r_1,\ldots,r_m$ in  $\{1,\ldots,\ell\}$ such that $r_0=p$, $r_m=q$ and $s_{r_{k-1},r_k}\neq 0$ for all $k\in \{1,\ldots,m\}$.
 By applying \eqref{e1}-\eqref{e2} with $(i,j)=(r_{k-1},r_k)$, we obtain
 \begin{equation}\label{suc}
 \frac{|\mathcal O_{r_{k-1}}|}{|\mathcal O_{r_k}|}\,=\,\frac{s_{r_k,r_{k-1}}}{s_{r_{k-1},r_{k}}}\,=\, \frac{|\mathcal Q_{r_{k-1}}|}{|\mathcal Q_{r_k}|}\,.
 \end{equation}
 Since (\ref{suc}) is true for all $k\in \{1,\ldots,m\}$, the telescoping products
  \begin{eqnarray*}
 && \frac{|\mathcal O_{r_0}|}{|\mathcal O_{r_1}|}
 \frac{|\mathcal O_{r_{1}}|}{|\mathcal O_{r_2}|}\,\ldots  \, \frac{|\mathcal O_{r_{m-1}}|}{|\mathcal O_{r_m}|}
   \;=\; \frac{|\mathcal O_{r_0}|}{|\mathcal O_{r_m}|}\;=\;
   \frac{|\mathcal O_p|}{|\mathcal O_q|}\,,\\[1.6mm]
   &&\frac{|\mathcal Q_{r_0}|}{|\mathcal Q_{r_1}|}
 \frac{|\mathcal Q_{r_{1}}|}{|\mathcal Q_{r_2}|}\,\ldots  \, \frac{|\mathcal Q_{r_{m-1}}|}{|\mathcal Q_{r_m}|}
 \;=\; \frac{|\mathcal Q_{r_0}|}{|\mathcal Q_{r_m}|}
 \;=\;\frac{|\mathcal Q_p|}{|\mathcal Q_q|}
 \end{eqnarray*}
  are equal.
This completes the proof of (\ref{nextstep}). For obtaining (\ref{claima}), it now suffices  to  observe that
$$\frac{|G|}{|\mathcal O_q|}=\frac{\sum_{p=1}^\ell |\mathcal O_p|}{|\mathcal O_q|}=\sum_{p=1}^\ell \frac{|\mathcal O_p|}{|\mathcal O_q|}=\sum_{p=1}^\ell \frac{|\mathcal Q_p|}{|\mathcal Q_q|}=\frac{\sum_{p=1}^\ell |\mathcal Q_p|}{|\mathcal Q_q|}=\frac{|H|}{|\mathcal Q_q|}$$
for all  $q\in \{1,\ldots,\ell\}$.
We have shown in this way that (c) implies  (d). The reverse implication is not true because the  equality  $\omega(G)= \omega(H)$ does not  ensure  orbital similarity between $G$ and $H$. Just think of the path $P_5$ and the house graph.  That (d) implies (e) is obvious from the  definition of graph entropy. However,  two connected graphs with the same entropy may have not only a different  orbit distribution vector, but also   a different  number of orbits, cf.\,Mowshowitz\,\cite[Figure\,3]{Mow}.
 \end{proof}

    \begin{figure}[!ht]
     \centering
    \includegraphics[width=0.5\textwidth]{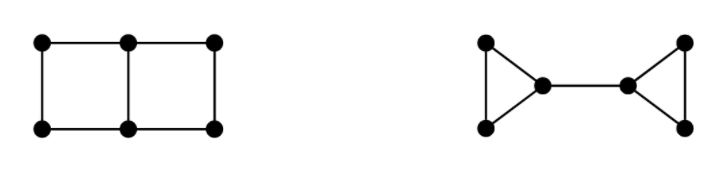}
    \caption{\small   Orbitally similar non-isomorphic graphs of the same order.}\label{Fig:domino}
    \end{figure}

\vskip 0,2cm

 Two connected graphs $G$ and $H$ are called  {\it orbitally homothetic} if  they have the same number of orbits, say $\ell$, and the corresponding orbits $\mathcal O_1,\ldots, \mathcal O_\ell$ and $\mathcal Q_1,\ldots, \mathcal Q_\ell$   can be labeled so that
    \begin{equation} \label{homothetie}
    \frac{\vert \mathcal Q_1 \vert}{\vert \mathcal O_1 \vert}= \ldots = \frac{\vert \mathcal Q_\ell \vert}{\vert \mathcal O_\ell \vert}\,.
    \end{equation}
The relation \eqref{nextstep} in the proof of
 Theorem\,\ref{th:implications} shows that orbital homothety is a consequence of orbital similarity. On the other hand, it is clear that orbital homothety is just the same property as (d).
 Some of the implications stated in Theorem\,\ref{th:implications} can be reversed if the pair $\{G,H\}$ enjoys additional hypotheses. We  do not  indulge on this matter because discussing such details is  space consuming and deviates from the central theme of our work. It is worthwhile mentioning however that the proof of the implication (c) $\Rightarrow$ (d)  provides a procedure for  computing the orbit distribution vector   from the orbit divisor matrix.

 \begin{example}  Consider a connected graph $G$ with orbit divisor matrix
$$S_G=\left[\begin{array}{ccc}
2&1&0\\3&0&1\\0&1&0
\end{array}\right].$$
We wish to determine the orbit distribution vector $\omega (G)=(\omega_1,\omega_2,\omega_3)$. To do this, we compute
\begin{align*}
&\omega_1 \,=\, \frac{|\mathcal O_1|}{|G|}\,=\,\left(\frac{|\mathcal O_1|}{|\mathcal O_1|}+\frac{|\mathcal O_2|}{|\mathcal O_1|}+\frac{|\mathcal O_3|}{|\mathcal O_1|}\right)^{-1}\,=\,\left(1+\frac{s_{1,2}}{s_{2,1}}+\frac{s_{2,3}\,s_{1,2}}{s_{3,2}\,s_{2,1}}\right)^{-1}\,=\,\frac{3}{5}\,,\\
&\omega_2\,=\,\frac{|\mathcal O_2|}{|G|}=\left(\frac{|\mathcal O_1|}{|\mathcal O_2|}+\frac{|\mathcal O_2|}{|\mathcal O_2|}+\frac{|\mathcal O_3|}{|\mathcal O_2|}\right)^{-1}\;\,=\,\left(\frac{s_{2,1}}{s_{1,2}}+1+\frac{s_{2,3}}{s_{3,2}}\right)^{-1}\,=\,\frac{1}{5}\,,\\
&\omega_3\,=\,\frac{|\mathcal O_3|}{|G|}\,=\,\left(\frac{|\mathcal O_1|}{|\mathcal O_3|}+\frac{|\mathcal O_2|}{|\mathcal O_3|}+\frac{|\mathcal O_3|}{|\mathcal O_3|}\right)^{-1}\,=\,\left(\frac{s_{2,1}\,s_{3,2}}{s_{1,2}\,s_{2,3}}+\frac{s_{3,2}}{s_{2,3}}+1\right)^{-1}\,=\,\frac{1}{5}\,.
\end{align*}
\end{example}

\vskip 0,4cm

For subsequent use, we state below some general  rules for constructing pairs of orbitally similar connected graphs. We restrict the  exposition   to some standard  graph operations: Cartesian product, strong product, corona product, vertex-coalescence,  and edge-coalescence.
\subsection{Cartesian, strong, and corona products}
 As usual, $K_n$ is the complete graph of order $n$.
The  Cartesian product operation on graphs is denoted with the symbol $\Box$. The rationale for this notation  is explained in the book of Hammack et al.\,\cite[Figure\,4.4]{HIK}.
The particular Cartesian   product
 $G^\diamond:= G\Box K_2$ is  called the   prism  built   on $G$ or the prism whose base is $G$.  It is possible to use in turn $G^\diamond$ as base for building a new prism, namely, $\diamond^{2}(G)= \left[G^\diamond\right]^\diamond$. This process can be iterated by using  the recursion formula
$\diamond^{r+1}(G)= \left[\diamond^{r}(G)\right]^\diamond $
initialized at  $\diamond^{1}(G)=G^\diamond$. We refer to $\diamond^{r}(G)$ as  the $r$-th prism built on $G$.

\begin{proposition} \label{pr:prism} Let $G$ and $H$ be orbitally similar connected graphs. Then:
\begin{itemize}
\item [(a)]  $G^\diamond $ and $H^\diamond$ are orbitally similar.
\item [(b)] More generally, $\diamond^{r}(G)$ and $\diamond^{r}(H)$ are orbitally similar for all $r\geq 1$.
\end{itemize}
\end{proposition}
\begin{proof}
{\it Part}\,(a). Let $G$ be a connected graph with vertices $v_1,\ldots, v_n$ and orbits  $\mathcal O_1,  \ldots, \mathcal O_\ell$.
   Let $G^\prime $ be a disjoint copy  of $G$ with vertices denoted by $ v_1^\prime,\ldots,  v_n^\prime$ and orbits  denoted by $\mathcal O_1^\prime,  \ldots, \mathcal O_\ell^\prime$.
    Blue vertices in $G^\prime$  are identified with blue vertices in $G$,   green vertices in $G^\prime$ are identified with   green  vertices in $G$, and so on.  The prism $ G^\diamond$ is obtained by connecting $v$ to $v^\prime$ for each $v\in V_G$.   For obvious reasons,  each  $\{v, v^\prime\}$ is called  a vertical edge of the prism.  Note that
    $G^\diamond$ is a connected graph with  twice as many vertices as $G$,
    but the number of orbits is still $\ell$.  Indeed, the orbit partition  of $G^\diamond$ is
    \begin{equation}\label{union} \{\mathcal O_1\cup \mathcal O_1^\prime,  \ldots, \mathcal O_\ell \cup\mathcal O_\ell^\prime\}.
    \end{equation}
  Since each vertical edge connects a pair of vertices of the same color,   it is clear that
    $ S_{G^\diamond} =S_{G} +I_\ell\,,$
    where $I_\ell$ is the identity matrix of order $\ell$. An analogous formula holds for the prism $H^\diamond$.  The assumption  $S_H= S_G$  yields then
    \begin{eqnarray}\label{chainette}
     S_{H^\diamond}\,=\,S_{H} +I_\ell\,=\, S_{G} +I_\ell\, =\, S_{G^\diamond}.
    \end{eqnarray}
 {\it Part}\,(b). By  using mathematical induction on $r$, we get $S_{\diamond^{r}(G)}= S_G +r I_\ell$, from where we derive the desired conclusion.
\end{proof}
 Incidentally, Proposition\,\ref{pr:prism}\,(a) admits a reverse formulation: if the prisms $G^\diamond $ and $H^\diamond$ are orbitally similar, then the corresponding bases $G$ and $H$ are orbitally similar. This can be seen by rewriting (\ref{chainette}) in the following order: $S_{H}=S_{H^\diamond} -I_\ell= S_{G^\diamond} -I_\ell = S_{G}$.  Figure\,\ref{Fig:prism} displays a pair  of orbitally similar prisms of the same order. We are using a planar representation in each case, so these graphs do not look as real life prisms. The distortion introduced by   planarity is because the vertical edges are not placed in a vertical position.

   \begin{figure}[!ht]
     \centering
    \includegraphics[width=0.52\textwidth]{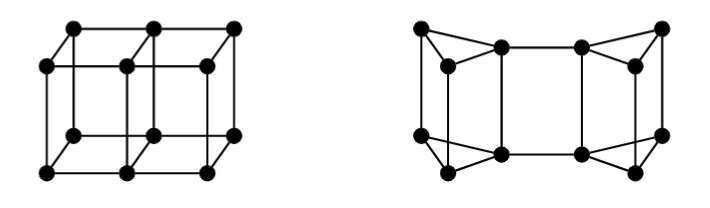}
    \caption{\small Orbitally similar  prisms built on the orbitally similar graphs of Figure\,\ref{Fig:domino}.  }\label{Fig:prism}
    \end{figure}

As said before, prisms are special Cartesian products. Strong prisms are obtained by using the  strong product operation $\boxtimes$ on graphs.  See the book \cite[p.\,36]{HIK} for the definition of the strong product $G\boxtimes F$ of two graphs. The particular strong product   $G^\boxtimes:= G\boxtimes K_2$  is  called the strong prism  built  on $G$ or the strong prism whose base is $G$.
Successive strong prisms  are obtained by using  the recursion formula
$\boxtimes^{r+1}(G)= \left[\boxtimes^{r}(G)\right]^\boxtimes $
initialized at  $\boxtimes^{1}(G)=G^\boxtimes$.

    \begin{proposition} \label{pr:strong} Let $G$ and $H$ be orbitally similar connected graphs. Then:
\begin{itemize}
\item [(a)]  $G^\boxtimes $ and $H^\boxtimes$ are orbitally similar.
\item [(b)] More generally, $\boxtimes^{r}(G)$ and $\boxtimes^{r}(H)$ are orbitally similar for all $r\geq 1$.
\end{itemize}
\end{proposition}
\begin{proof}  {\it Part}\,(a). We use the same notation as in the proof  of Proposition\,\ref{pr:prism}\,(a). For constructing   $G^\boxtimes$, we start by considering   the base $G$ and  the  disjoint copy $G^\prime$. Then, for each  $v\in V_G$, we add not only the vertical edge   $\{v, v^\prime\}$, but also  a group of $d_G(v)$ crossed edges, namely, $\{v, u^\prime\}$ with  $u\in N_G(v)$. The crossed edges are added  in the strong prism, but not in the usual prism. Hence,  $G^\boxtimes$  admits  $G^\diamond$ as a proper subgraph. The orbit partition of $G^\boxtimes$ is however the same as that of $G^\diamond$, namely, (\ref{union}). It is not difficult  to check that
\begin{equation} \label{formo} S_{G^\boxtimes} = 2S_{G} +I_\ell.
\end{equation}
The identity matrix  in (\ref{formo}) is due to  the  vertical edges, whereas the factor $2$ in front of $S_{G}$ is due  to the crossed edges.
Since $S_G=S_H$, we get
$S_{G^\boxtimes}=2S_{G} +I_\ell= 2S_{H} +I_\ell= S_{H^\boxtimes}$.
{\it Part}\,(b).  Mathematical induction on $r$ yields $S_{\boxtimes^{r}(G)}= 2^rS_G + \left(2^r-1\right) I_\ell.$ This formula leads to the desired conclusion.
\end{proof}

 \begin{figure}[!ht]
     \centering
    \includegraphics[width=0.5\textwidth]{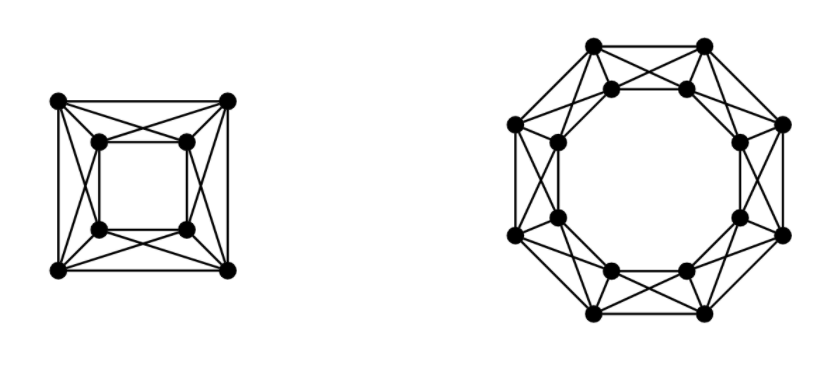}
    \caption{\small   Strong prisms that are orbitally similar. }\label{Fig:strong}
    \end{figure}

The strong prisms shown in Figure\,\ref{Fig:strong} are orbitally similar  because the corresponding bases $C_4$ and $C_8$ are orbitally similar. The symbol $C_n$ denotes the cycle of order $n$.   Instead of a Cartesian product or a strong product, the  next proposition uses a non-commutative  operation
$(G,F)\mapsto G\odot F$
introduced in Frucht and Harary\,\cite{FH} and defined  as follows.
 The corona $G\odot F$  of two graphs $G$  and $F$ is the graph
obtained by taking $\vert G\vert$ copies of $F$, and then joining the $i$-th vertex of $G$ to every vertex in the $i$-th copy of $F$. Of special interest is the particular case  $G^\odot:= G\odot K_1$.
 Since  $K_1$ is the smallest possible choice as ingredient $F$,  the graph $G^\odot$ is called the minimal corona    built on $G$. For instance, $C_n^\odot$ corresponds to the sun graph on $2n$ vertices.  Of course, it  possible to build  in turn a minimal corona  on $G^\odot$.  This process can be iterated by using  the recursion formula
$\odot^{r+1}(G)= \left[\odot^{r}(G)\right]^\odot $
initialized at  $\odot^{1}(G)=G^\odot$. We refer to $\odot^{r}(G)$ as the $r$-th minimal  corona built on $G$.

\begin{proposition} \label{pr:corona} Let $G$ and $H$ be orbitally similar connected graphs. Then:
\begin{itemize}
\item [(a)] $G^\odot $ and $H^\odot$ are orbitally similar.
\item [(b)] More generally, $\odot^{r}(G)$ and $\odot^{r}(H)$ are orbitally similar for all $r\geq 1$.
\end{itemize}
\end{proposition}
\begin{proof} It suffices to prove the first part, because the second one follows afterward by mathematical induction on $r$.
Let $G$ be a connected graph with vertex set $V= \{v_1,\ldots, v_n\}$ and orbit partition   $\{\mathcal O_1,  \ldots, \mathcal O_\ell\}$.   The vertex set of $G^\odot$ is
$ V\cup V^\prime$,  where  $V^\prime= \{v_1^\prime, \ldots, v_n^\prime\}$ and $v^\prime _i$ is the vertex of the $i$-th copy of $K_1$ that is adjacent to $v_i$ in the construction of $G^\odot$.
 Let $k\in\{1,\ldots,\ell\}$.  We claim that  $\mathcal O_k$ is an orbit of $G^\odot$. For any pair $u,v$ in $\mathcal O_k$, there exists  an automorphism $\varphi$ of $G$ such that $\varphi(u)=v$.  For  $i\in \{1,\ldots,n\}$, let $\tilde{\varphi}(v_i)=\varphi(v_i)$ and  $\tilde \varphi(v^\prime_i)$ be the vertex in $V^\prime$ that is adjacent to $\varphi(v_i)$.
 Then   \color{black} $\tilde \varphi$  is an automorphism of $G^\odot$ such that  $\tilde{\varphi}(u)=v$. On the other hand,
$V^\prime\cap \mathcal O_k=\emptyset$ because   each vertex in $V^\prime$ has degree $1$ and each  vertex in $\mathcal O_k$ has degree at least two. This completes the proof of the claim. A similar argument shows that  also $\mathcal O^\prime_k:=\{v^\prime:v\in\mathcal O_k\}$
 is an orbit of $G^\odot$. In fact, $\mathcal O_1,\ldots,\mathcal O_\ell,\mathcal O^\prime_1,\ldots,\mathcal O^\prime_\ell$ is the orbit partition of $G^\odot$. Once the orbits of $G^\odot$ have been identified,  it is easy  to build  the corresponding orbit divisor matrix. We get
$$S_{G^\odot}= \left[
                \begin{array}{cc}
                  S_G & I_\ell \\
                  I_\ell & O_\ell \\
                \end{array}
              \right],$$
where $O_\ell$ is the zero matrix of order $\ell$. An analogous formula holds for $S_{H^\odot}$. Note that $S_{G^\odot}=S_{H^\odot}$ because $S_G=S_H$.
\end{proof}

    \subsection{Vertex-coalescence and edge-coalescence}
 The next proposition uses a non-commutative binary operation
$ (G,L)\mapsto G\nabla L$   called vertex-loading (or loading operation, in short).
The loading  of a rooted graph $L$ onto a graph $G$ consists in coalescing  (or gluing) a copy   of $L$ at each vertex of $G$. The coalescence  is   of course  relative to the root vertex of $L$.
By economy of language, we say that $G\nabla L$ is the graph ``$G$ loaded with $L$''. It is  reasonable to call   $G$ and $L$  the support and the load, respectively. For ease of  presentation, we consider as support $G$ a connected graph  with a very simple orbital structure, say a vertex-transitive connected graph.  As load $L$, we consider a connected graph with  possibly many orbits, but with an orbital structure that is  not too involved. The prototype example of load that we have in mind is 
\begin{equation} \label{Lqm}
 L_{q,m}:= \left\lbrace
\begin{array}{l}
q \mbox{ copies of a path }P_{m+1} \mbox{ of  length } m\\
\mbox{coalescing at a common end-vertex}.
\end{array}
\right.
 \end{equation} 
 The particular case $L_{1,m}$ corresponds to  a path of length  $m$ rooted at one of its end-vertices. Although $L_{2,m}$  is a path of length $2m$, we prefer to view this graph as a coalescence of two paths  of length $m$. If $q\geq 3$, then  $L_{q,m}$ is an equilibrated starlike tree: it is  obtained by coalescing,  at a common vertex,  $q$ copies  of $P_{1+m}$, cf.\,Figure\,\ref{Fig:loading}.   The common vertex is declared the root of $L_{q,m}$. Note that $L_{q,m}$ has $1+qm$ vertices.  In particular,  $L_{q,1}=K_{1,q}$ is a star of order $1+q$.

\begin{proposition} \label{pr:load} Let $G$ and $H$ be vertex-transitive  connected graphs of the same degree.  Then:
\begin{itemize}
\item [(a)] $G$ and $H$ are orbitally similar.
\item [(b)] $G\nabla L_{q,m}$ and $H\nabla L_{q,m}$ are orbitally similar for all  $m,q\geq 1$.
\end{itemize}
\end{proposition}
\begin{proof}  {\it Part} (a).  Since $G$ is a vertex-transitive connected graph, its orbit divisor  matrix $S_G$ is  a scalar, namely,  the degree of $G$. Since  $H$ is a vertex-transitive connected graph with the same degree as $G$, we have $S_G=S_H$. {\it Part} (b). Let $n$ and $d$ be the order and degree of $G$, respectively.  Let $S$ be the orbit divisor  matrix
 of  $G\nabla L_{q,m}$. We shall derive an explicit formula for $S$  and see
that this matrix depends on $q$, $m$, and $d$,  but not on $n$.  Clearly,  $G\nabla L_{q,m}$ has $m+1$ orbits.  The first orbit $\mathcal O_1$ is formed with   the $nq$ vertices at distance $m$ from the support $G$, the second orbit $\mathcal O_2$ is formed  with the $nq$   vertices at distance $m-1$ from $G$, and so on. The last  orbit $\mathcal O_{m+1}$ is the only one  of different size. It is formed with the $n$ vertices   on the support $G$. If the orbits of $G\nabla L_{q,m}$ are labeled as just mentioned, then
\begin{equation} \label{Sloaded}
S=  \left[
   \begin{array}{cc}
     A_{P_m} & e_{m} \\
     qe_{m}^\top & d \\
   \end{array}
 \right]\,,
\end{equation}
where $A_{P_m}$ is the adjacency matrix of the path $P_m$ and $e_m$ is the last canonical vector of
$\mathbb{R}^{m}$. The matrix $S$ depends on $m$, $q$ and $d$, but  it  is independent of $n$. In other words, the order of $G$ is irrelevant in the construction of $S$. Since $H$ is also of degree $d$,
the  orbit divisor  matrix
of  $H\nabla L_{q,m}$ is also expressible as in  (\ref{Sloaded}).
\end{proof}

 \begin{figure}[!ht]
     \centering
   \includegraphics[width=0.62\textwidth]{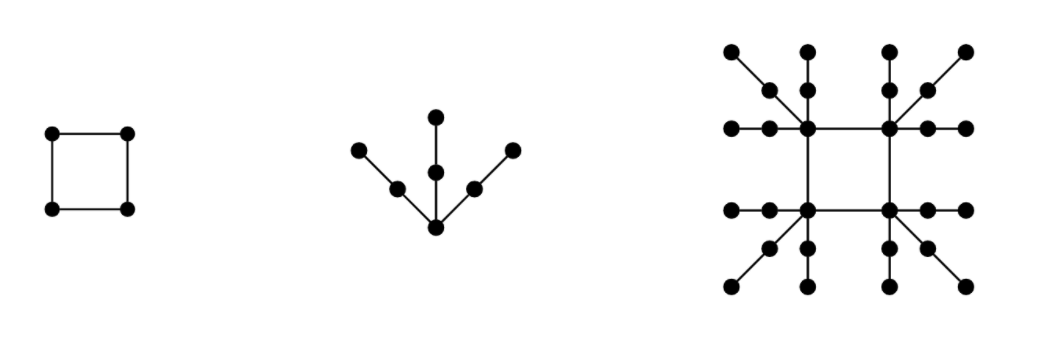}\put(-79,16){\scriptsize$\nabla$}\put(-42,16){\scriptsize$=$}
    \caption{\small   Vertex-loading of  the equilibrated starlike tree $L(3,2)$ onto the cycle $C_4$. }\label{Fig:loading}
    \end{figure}

The operation of vertex-loading has an edge-loading  counterpart
$ (G,L)\mapsto G\nabla^e L$.  The superscript ``$e$'' in the symbol $\nabla^e$ underlines the fact that a copy of $L$ is
loaded on  each edge of the support $G$.  This means that $L$ is considered as an edge-rooted graph and that its distinguished  edge is glued at each edge of $G$. As prototype example of edge-rooted load we consider
\begin{equation} \label{Bqm}
 B_{q,m}:= \left\lbrace
\begin{array}{l}
q \,\mbox{ copies of a cycle }C_{m} \\
\mbox{sharing a common edge}.
\end{array}
\right.
 \end{equation}
  Such a connected graph is known as the $m$-gonal book with $q$ pages. The common edge, which  is called the spine of the book,  serves as distinguished edge. By definition,   $G\nabla^e B_{q,m}$ is the graph obtained by loading each edge of  $G$  with a copy of the book $B_{q,m}$. Metaphorically speaking, each edge of the  support (or  ``shelve'') is loaded  with a copy of the same book, cf.\,Figure\,\ref{Fig:book}.
 The edge counterpart of Proposition\,\ref{pr:load}\,(b) reads as follows.
 \begin{proposition} \label{pr:edgeload} Let $G$ and $H$ be vertex-transitive  connected graphs of the same degree.   In addition, suppose that $G$ and $H$ are edge-transitive.
 Then
 $G\nabla^e B_{q,m}$ and $H\nabla^e B_{q,m}$ are orbitally similar for all  $m\geq 3$ and $q\geq 1$.
\end{proposition}
\begin{proof}
Let $n$ and $d$ be the order and degree of $G$, respectively. If  $d=0$, then   $G$ has no edge and the results holds vacuously. If  $d=1$, then  $G$ and $H$ are isomorphic to $K_2$ and we are done.  We suppose then that $d\geq 2$. Let $S$ be the orbit divisor  matrix of $G\nabla^e B_{q,m}$. Since $G$ is both vertex-transitive and edge-transitive and  $B_{q,m}$ has a relatively simple orbital structure, it is possible to derive an explicit formula for $S$.  The number of orbits of  $G\nabla^e B_{q,m}$ is equal to the upper integer part of $m/2$.
Up to permutation similarity transformation, the explicit form of $S$ is
\begin{eqnarray*}
\left[
    \begin{array}{cc}
        0&  2\\
       dq &  d\\
     \end{array}
   \right],\; \left[
     \begin{array}{cc}
       1 &  1\\
       dq & d \\
     \end{array}
   \right], \; \left[
     \begin{array}{ccc}
     0& 2 &0\\
       1 &  0& 1\\
      0& dq & d \\
     \end{array}
   \right], \;\left[
     \begin{array}{ccc}
     1& 1 &0\\
       1 &  0& 1\\
      0& dq & d \\
     \end{array}
   \right]
\end{eqnarray*}
for  $m=3, 4, 5$, and $6$, respectively.
The general pattern is clear.   Let  $r:= \lceil m/2\rceil -1$ and $\{e_1,e_2,\ldots, e_r\}$ be the canonical basis of $\mathbb{R}^r$. Depending on whether $m$ is odd or even, we get
$$
  \left[
   \begin{array}{cc}
     A_{P_r}+e_1e_1^\top & e_{r} \\
     dqe_{r}^\top & d \\
   \end{array}
 \right],  \;\left[
   \begin{array}{cc}
     A_{P_r}+e_1e_2^\top & e_{r} \\
     dqe_{r}^\top & d \\
   \end{array}
 \right]\,,
$$
respectively.
 In either case, the matrix  $S$   is independent of $n$. The orbit divisor  matrix of
 $H\nabla^e B_{q,m}$ is then  the same as that of  $G\nabla^e B_{q,m}$.
\end{proof}
\vskip 0,2cm
\begin{remark} Edge-transitivity is an essential assumption in Proposition\,\ref{pr:edgeload}. The circular ladders ${\rm CL}_3= C_3^\diamond$ and ${\rm CL}_4= C_4^\diamond$ are orbitally similar vertex-transitive connected graphs. However, the edge-loaded graphs ${\rm CL}_3\nabla^e B_{1,3}$ and ${\rm CL}_4\nabla^e B_{1,3}$ are not orbitally similar, because the first one has  $3$ orbits and the second one has only $2$ orbits. Observe that ${\rm CL}_4$ is edge-transitive, but ${\rm CL}_3$  is not edge-transitive.
\end{remark}

 \begin{figure}[!ht]
     \centering
   \includegraphics[width=0.62\textwidth]{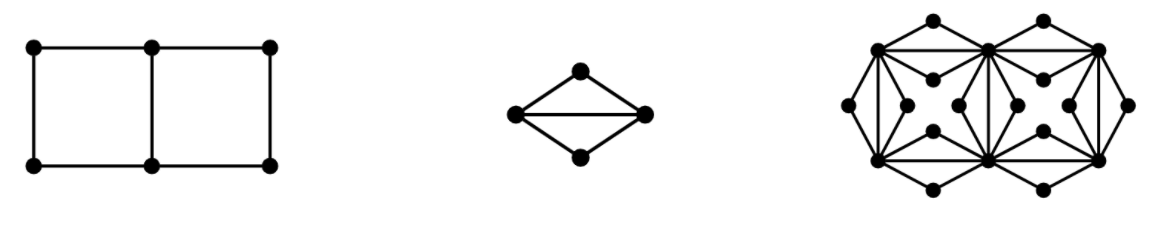}\put(-69,10){\scriptsize$\nabla^e$}\put(-38,10){\scriptsize$=$}
    \caption{\small   Edge-loading of   $B_{2,3}$ onto the domino graph $P_3\Box P_2$. }\label{Fig:book}
    \end{figure}

\section{Construction of  self-similar sequences} \label{se:construct}

 This section develops some  techniques for constructing self-similar sequences. A self-similar sequence already constructed can be used in turn for obtaining new ones.  For ease  of presentation, we do not pay attention to the seed from where the self-similar sequence emanates.    Parts (a) and (b) in Proposition\,\ref{pr:c} are obvious.    Part (c) follow from Proposition\,\ref{pr:prism}, Proposition\,\ref{pr:strong} and  Proposition\,\ref{pr:corona}.

\begin{proposition} \label{pr:c} Let $\{G_k: k\geq 1\}$ be a self-similar sequence.
Then:
\begin{itemize}
\item [(a)] Each subsequence of  $\{G_k: k\geq 1\}$ is  self-similar.
\item [(b)] If each $G_k$ is   exchanged   by an orbital similar connected graph $H_k$ such that $\vert H_k\vert =\vert G_k\vert$, then the new sequence $\{H_k: k\geq 1\}$ is self-similar.
\item [(c)] $\{G_k^\ast: k\geq 1\}$ is  a  self-similar sequence, where  $\,\ast\,$ is either the prism operation $\diamond$,  the strong prism operation $\boxtimes$ , or the minimal corona operation $\odot$.
\end{itemize}
\end{proposition}

\subsection{Using vertex-transitivity}
The first example of self-similar sequence that comes to mind  is
\begin{equation}\label{GC}
 \mathbb {G}^{\rm C}:=\{C_{n}: n\geq 3\} \,.
 \end{equation}
 Passing from $C_n$ to  $C_{n+1}$ increases the number of vertices, but  does not modify the  orbit divisor matrix. The explanation is as follows. As mentioned before,
if $G$ is a vertex-transitive connected graph, then $S_G$ is  a scalar, namely, the degree of $G$.
 Since any cycle is a vertex-transitive connected graph of degree $2$, we have
$S_{C_n}=2 $ for all $ n\geq 3$. Hence, (\ref{GC}) satisfies the orbit similarity condition (\ref{c2}).  Other examples of
self-similar sequences in the same vein are
\begin{eqnarray}\label{CL}
 \mathbb {G}^{\rm CL}&:=&\{{\rm CL}_n: n\geq 3\},\\ \label{ML}
 \mathbb {G}^{\rm ML}&:=&\{{\rm ML}_n: n\geq 3\}, \\ \label{CP}
 \mathbb {G}^{\rm CP}&:=&\{{\rm CP}_n: n\geq 3\},\\ \label{AP}
 \mathbb {G}^{\rm AP}&:=&\{{\rm AP}_n: n\geq 3\}.
\end{eqnarray}
 The graph ${\rm CL}_n:=C_n^\diamond $  is the circular ladder of order $2n$. Since ${\rm CL}_n$ is a prism with circular base, the self-similarity of (\ref{CL})  is a consequence of Proposition\,\ref{pr:prism} and the fact that (\ref{GC}) is self-similar.   ${\rm ML}_n $ denotes the  Moebius ladder of order $2n$  and ${\rm CP}_n $ is the  crossed prism  of order $2n$.  Crossed prisms are not to be confused with strong prisms.   Each Moebius ladder and each crossed prism is a vertex-transitive connected graph of degree $3$. This observation and Proposition\,\ref{pr:vt} (stated in a moment) explain why  (\ref{ML}) is self-similar and why (\ref{CP}) is self similar.
 ${\rm AP}_n$  stands for  the anti-prism graph of order  $2n$. Each  anti-prism is a vertex-transitive connected graph of degree  $4$. The self-similarity of (\ref{AP}) is also taken care by  Proposition\,\ref{pr:vt}.

 \begin{proposition} \label{pr:vt}
Let  $\mathbb{G}=\{G_k: k\geq 1\}$ be  a sequence of connected graphs on an  increasing number of vertices.
\begin{itemize}
\item [(a)] If  the $G_k$'s are vertex-transitive of the same  degree, then $\mathbb{G}$ is self-similar.
\item [(b)] Conversely, if $\mathbb{G}$ is  self-similar and
 at least one of the $G_k$'s is vertex-transitive,  then all the $G_k$'s are vertex-transitive of the same  degree.
\end{itemize}
\end{proposition}
\begin{proof} Part (a) is clear: if $d$ is the common degree of the $G_k$'s, then  $S_{G_k}=d$
for all $k\geq 1$ and  condition (\ref{c2}) is in force.  For proving (b), we suppose that there exists an integer $k_0\geq 1$ such that $G_{k_0}$ is  vertex-transitive.  Let $d_0$ be  the degree of  $G_{k_0}$. Hence,  each $G_k$ has exactly one orbit and
$S_{G_k}= S_{G_{k_0}}= d_0$.
In conclusion, each $G_k$ is   vertex-transitive of degree $d_0$.
\end{proof}

A word of caution is in order: vertex-transitivity alone is not enough to ensure self-similarity.  For instance, each complete graph $K_n$ is vertex transitive, but  the sequence $ \{K_{n}: n\geq 1\}$  is not  self-similar. This is  because $S_{K_n}=n-1 $ depends on $n$.   The next result  is easy. 
\begin{proposition} \label{pr:vtbis} For  each integer \color{black} $d\geq 2$, there exists a self-similar sequence
 of vertex-transitive graphs of degree $d$.
\end{proposition}
\begin{proof} If $d=2$, then we  take the sequence $\{C_n:n\geq 3\}$ of cycles. Consider then the case   $d\geq 3$. Thanks to
Proposition\,\ref{pr:prism}\,(b), the sequence
$\{\diamond^{d-2}(C_n):n\geq 3\}$
is  self-similar. Note that $\diamond^{d-2}(C_n)$ is a vertex-transitive graph of degree $d$, as requested.
\end{proof}
In a sense, a torus grid graph is a generalization of a cycle. It is clear that,
for all  $n\geq 3$ and $m\geq 3$, the torus grid  $T(n,m):=C_n\Box C_m$  is a vertex-transitive connected graph of degree $4$. Hence,
$S_{T(n,m)}= 4$ is a scalar independent of $n$ and $m$. If we fix one parameter of $T(n,m)$, then  we obtain a self-similarity sequence indexed by the other parameter.
\begin{proposition}\label{pr:torus}
For each integer $m\geq 3$, the sequence $\{{\rm T}(n,m): n\geq 3\}$
is self-similar.
\end{proposition}
This easy result can be generalized in many ways. For instance, we may perturb simultaneously $n$ and $m$, but making sure that the product $nm$ increases.
\begin{proposition}\label{pr:interm}
Let  $n_k, m_k\geq 3$ be integers such that $n_km_k$ increases with  $k\geq 1$. Then  the sequence $\{{\rm T}(n_k,m_k): k\geq 1\}$
of torus grids is self-similar.
\end{proposition}
We skip the  proof of Proposition\,\ref{pr:interm} because it is integrated in the proof of Theorem\,\ref{th:majo}.  Another extension of Proposition\,\ref{pr:torus} consists in  working with an $r$-dimensional  torus graph
\begin{equation} \label{multitorus}
 T(s_1, s_2, \ldots, s_r):= C_{s_1}\Box C_{s_2}\Box \ldots \Box  C_{s_r}.
 \end{equation}
 For convenience, a graph like  (\ref{multitorus}) is called a  {\it multi-parametric torus }  and it is denoted by $T(\vec{s}\,)$. By abuse of language, we say that $r$ is the dimension of (\ref{multitorus}).
If we fix for instance the  parameters $s_2,\ldots, s_r\geq 3$, then we  get  a self-similar sequence indexed by $s_1\geq 3$. There are still  other possibilities of  generalization. We may think for instance of loading each vertex of  (\ref{multitorus}) with a prescribed load, say  $L_{q,m}$. We get in this way the following   self-similarity theorem for loaded multi-parametric torii.  For the sake of completeness, Theorem\,\ref{th:majo}  considers also the choice  $r=1$,  in which  case the multi-parametric torus degenerates  into a cycle.

\begin{theorem} \label{th:majo}
Let $q,m$, and $ r$ be  fixed positive integers. Let $\vec{s}(k)=(s_1^k, s_2^k, \ldots, s_r^k)$ be a   multi-index, with  $s_1^k, s_2^k, \ldots, s_r^k\geq 3$, such that the product $s_1^k s_2^k \ldots s_r^k$ increases  as
$k$ increases.  Then
$$ \{T(\vec{s}(k))\nabla L_{q,m}: k\geq 1\}$$
is a self-similar sequence.
\end{theorem}
\begin{proof}   Our first observation is that each $r$-dimensional  torus
$G_k:=T(\vec{s}(k))$
is a connected graph. The second observation is that $$\vert G_k\vert =s_1^k s_2^k \ldots s_r^k $$
increases with $k$. The third observation is the key point of the proof: the  $G_k$'s are  vertex transitive graphs of the same degree. Indeed, all of them have degree  $2r$. Proposition\,\ref{pr:vt}\,(a) ensures that $\mathbb{G}:=\{G_k:k\geq 1\}$ is a self-similar sequence. The final step consists in loading each $G_k$ with the same load, namely, $L_{q,m}$. Proposition\,\ref{pr:load}\,(b) and the self-similarity of $\mathbb{G}$ complete  the job.
\end{proof}

\subsection{Using bi-orbitality}
 We now pass to   graphs with  $2$ orbits. \color{black}   A bi-orbital connected graph
 is less symmetric than a vertex-transitive connected graph, but it  has still a good deal  of symmetry.  The vertex set of a bi-orbital connected graph $G$ is partitioned into blue vertices and  red vertices.
The orbit divisor matrix
$$ S_G= \left[
          \begin{array}{cc}
            s_{1,1} &  s_{1,2}\\
            s_{2,1} & s_{2,2} \\
          \end{array}
        \right] = \left[
          \begin{array}{cc}
            s_{\rm blue, \rm blue} &  s_{\rm blue, \rm red}\\
             s_{\rm red, \rm blue}&  s_{\rm red, \rm red}\\
          \end{array}
        \right]$$
 is of order two and, therefore,  it is easy to handle.  Since $G$ is  connected, we  have
 $s_{1,1}\geq 0$, $s_{2,2} \geq 0$, $s_{1,2}\geq 1$, and $ s_{2,1} \geq 1$.
One  quickly realizes that, for each integer $q\geq 1$,   there are infinitely many graphs with
 $$\left[
          \begin{array}{cc}
            0&  1\\
            q & 2 \\
          \end{array}
        \right] $$
         as orbit divisor  matrix; think for instance of a sun graph (case $q=1$) or a generalized sun graph (case $q\geq 2$).
       \begin{figure}[!ht]
     \centering
    \includegraphics[width=0.7\textwidth]{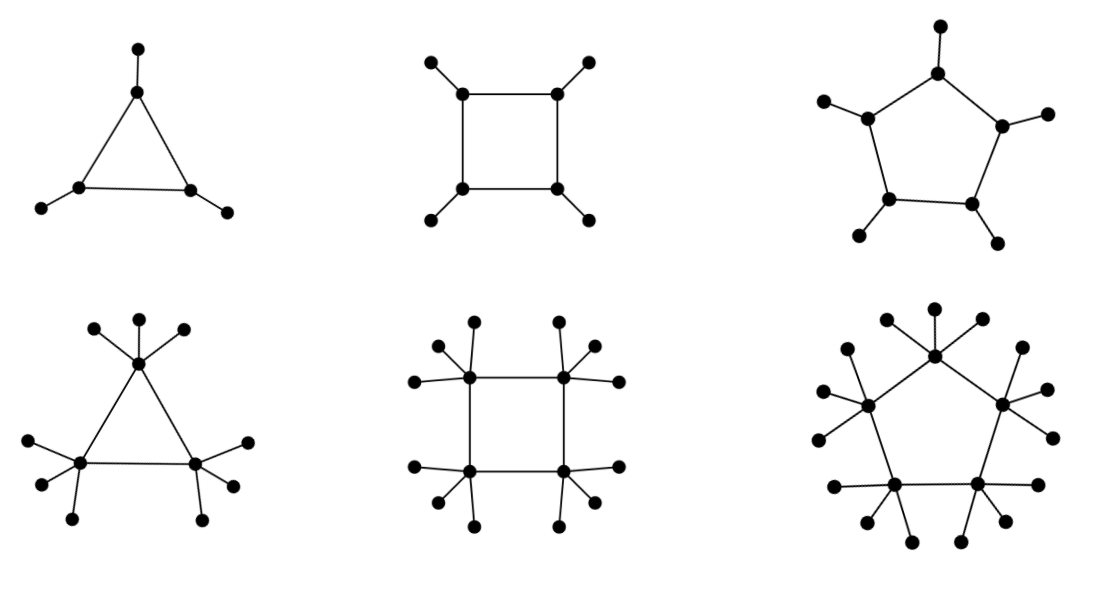}
    \caption{\small   Self-similar sequence of sun graphs (first row) and generalized sun graphs (second row).  }\label{Fig:sun}
    \end{figure}
         In the second row of Figure\,\ref{Fig:sun}, we display a self-similar sequence of generalized sun graphs. The parameter $q$ represents the number of rays emanating from each vertex of a cycle (we are taking $q=3$ for ease of  visualization). Such  elementary examples lead to  more elaborate constructions in the same vein.  Instead of attaching  vertices to a cycle as in a generalized sun graph, we may consider coalescing one or several  copies of a complete graph at each vertex of a cycle.   This idea is formalized in the next proposition.    With the particular choice $p=2$ in  Proposition\,\ref{pr:Gnpq}\,(a),  we recover  the self-similar sequence of generalized sun graphs.  The notation
$qK_p$ in part (b) refers to $q$ disjoint copies of $K_p$.
\begin{proposition} \label{pr:Gnpq}
Let $p$ and $q$ be positive integers.
\begin{itemize}
\item [(a)] For each $n\geq 3$,  let $G_n(p,q)$ be the graph obtained by coalescing $q$  copies of a complete graph $K_p$ at each vertex of a cycle $C_{n}$.
Then $\{G_n(p,q):n\geq 3\}$ is a self-similar sequence.
\item [(b)]  More generally, consider a sequence $\{F_k:k\geq 1\}$ of vertex-transitive connected graphs of prescribed  degree, say $d\geq 2$. If $\vert F_k\vert$ grows with $k$, then
    $ \{F_k\odot (qK_p): k\geq 1\}$
    is a self-similar sequence.
\end{itemize}
\end{proposition}
\begin{proof} {\it Part} (a). We suppose that $p\geq 2$, otherwise $G_n(p,q)$ is a cycle of order $n$ and we are done. The graph $G_n(p,q)$ is clearly connected and bi-orbital. After the coalescence has taken place, we paint with red the $n$ vertices on $C_n$ and with blue all the remaining vertices. The situation is as follows: each blue vertex is adjacent to $p-2$ blue vertices and adjacent to $1$ red vertex; on the other hand, each red vertex is adjacent to $q(p-1)$ blue vertices and adjacent to $2$ red vertices.
Note that the orbit divisor matrix
$$ S_{G_n(p,q)}=\left[
          \begin{array}{cc}
            p-2&  1\\
            q(p-1) & 2 \\
          \end{array}
        \right]$$
of $G_n(p,q)$ does not depend on $n$. {\it Part (b)}. The proof is analogous. The order of $G_k:=F_k\odot (qK_p)$ increases with $k$ but
$$ S_{G_k}=\left[
          \begin{array}{cc}
            p-1&  1\\
            qp & d \\
          \end{array}
        \right]$$
is independent of $k$.
\end{proof}

               \begin{figure}[!ht]
     \centering
    \includegraphics[width=0.8\textwidth]{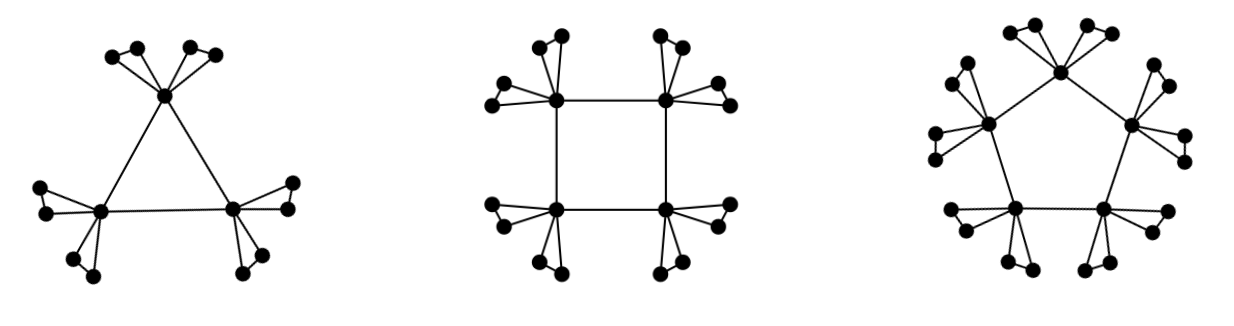}
    \caption{\small   Self-similar sequence of bi-orbital graphs  as in Proposition\,\ref{pr:Gnpq}\,(a).  }\label{Fig:biorbital}
    \end{figure}

\section{Preservation of various  graph invariants}\label{se:preserve}
As a consequence of Theorem\,\ref{th:implications},  a self-similar sequence $\{G_k:k\geq 1\}$ emanating from a connected graph $G$ has the property that ${\rm Ent}(G_k)= {\rm Ent}(G)$
for all $k\geq 1$. Entropy preservation is almost  tautological from the way  self-similarity  has been introduced. What is more striking perhaps is that  self-similarity preserves  also many other fundamental graph invariants.
\subsection{Spectral radius preservation}
 The spectral radius is one of the most important structural parameters of a graph.
We find it particularly interesting  the fact that  a self-similarity sequence preserves the spectral radius of the seed.   Such a result is a  consequence of the next lemma.  
\begin{lemma}\label{le:nice}
 Let $G$ be a connected graph.   Then:
 \begin{itemize}
 \item [(a)] Each  eigenvalue of $S_G$  is  real. In fact, it belongs to the spectrum of  $A_G$.
  \item [(b)] $\varrho(G)$ is equal to the largest eigenvalue of $S_G$.
 \end{itemize}
\end{lemma} \color{black} 
\begin{proof} 
 Part (a) \color{black} is a particular case of a more general result concerning the divisor matrix associated to an equitable partition,  cf.\,Barrett at al.\,\cite[Theorem\,3.2]{BFW}.   Part (b) \color{black} is obtained as a consequence of Cvetkovi\'{c} et al.\,\cite[Corollary 2.4.6]{CRS}, but, for the sake of the exposition, we give a short proof.  For alleviating notation, we use  $V_G=\{1,\ldots,n\}$ as vertex set. The orbit partition
 $$ V_G= \underbrace{\{i^{(1)}_1,\ldots,i^{(1)}_{m_1}\}}_{\mathcal O_1} \;\cup \;\underbrace{\{i^{(2)}_1,\ldots,i^{(2)}_{m_2}\}}_{\mathcal O_2}\; \cup \ldots \cup \;\underbrace{\{i^{(\ell)}_1,\ldots, i^{(\ell)}_{m_\ell}\}}_{\mathcal O_\ell}$$
 is made of  $m_1$ blue vertices, $m_2$  green  vertices, and so on.
Let $A_G= [a_{i,j}]$ be the adjacency matrix of $G$ and $x$ be its principal eigenvector. So, the components of  $x$ are   positive, sum up to $1$,  and  $A_Gx=\varrho(G)x$. Thanks to Proposition\,\ref{pr:perron},  $x$ has the special form
 \begin{equation}\label{perron}
 x= (\alpha_1,\ldots,\alpha_1,\alpha_2,\ldots,\alpha_2, \ldots,\alpha_\ell,\ldots,\alpha_\ell)^\top,
 \end{equation}
 where the component (or color)  $\alpha_i$ appears $m_i$ times. We claim that
 \begin{equation} \label{bigclaim} S_G\alpha=\varrho(G)\alpha\,,
 \end{equation}
  where $\alpha:=(\alpha_1,\ldots,\alpha_{\ell})^\top$.
  If such a claim is true, then $\varrho(G)$ is an eigenvalue of $S_G$ and (b) follows. The proof of (\ref{bigclaim}) relies on the fact that $\sum_{j\in \mathcal O_q}a_{i,j}=s_{p,q}$
  for all $p,q\in\{1,\ldots,\ell\}$ and   $i\in \mathcal O_p$. These relations are just a rewriting of (\ref{defSG}) with obvious changes in notation.
 Thus, for $p\in\{1,\ldots,\ell\}$ and $i\in \mathcal O_p$,  we get
 \begin{eqnarray}\label{sum}
 (S_G\alpha)_p&=&\sum_{k=1}^\ell s_{p,k}\alpha_k=\sum_{k=1}^\ell \sum_{j\in \mathcal O_k}a_{i,j}\alpha_k
 =\sum_{j=1}^n a_{i,j}x_j\\[1.6mm]\label{sumbis}
 &=&(A_Gx)_i=\varrho(G) x_i= \varrho(G) \alpha_p= (\varrho(G)\alpha)_p\,.
 \end{eqnarray}
The last equality in (\ref{sum}) is due to \eqref{perron} and  the block structure of $A_G$, the remaining equalities are clear.
 Since $p$ is arbitrary in (\ref{sum})-(\ref{sumbis}), we  get (\ref{bigclaim})
 as desired.
\end{proof}

We now are ready to state the announced spectral radius preservation result. 
\begin{theorem} \label{th:preserverho} Let $G$ be a connected graph. Then
any connected graph  that is orbitally similar to $G$ has the same spectral radius as $G$.
In particular, if  $\{G_k:k\geq 1\}$ is  a self-similar sequence emanating from  $G$, then
$\varrho(G_k)= \varrho(G)$
for all $k\geq 1$.
\end{theorem}
\begin{proof} This result   follows  from Lemma\,\ref{le:nice}. Indeed, if   $G$ and $H$ are orbitally similar  connected graphs, then
 $ \varrho(G)= \lambda_{\rm max}(S_G)=\lambda_{\rm max}(S_H) =\varrho(H)$.  The way we label the orbits of $G$ and $H$ is  irrelevant.
\end{proof}

  The condition
 $\omega(G)=  \omega(H)$ alone does not suffice to ensure that  $G$ and $H$ have the same spectral radius. In fact,  the difference between $\varrho(G)$ and $\varrho(H)$ can be made as large as one wishes.  For instance, if  we compare the spectral radiuses of the vertex-transitive graphs $C_n$ and $K_n$, we see that $\varrho(C_n)=2$, whereas  $\varrho(K_n)=n-1$ goes to infinity with $n$.

\subsection{Preservation of other graph invariants}
Spectral radius preservation of the seed is a remarkable feature  of self-similarity. Less astonishing perhaps is the preservation of the minimum degree $\delta(G)$, the maximum degree $\Delta(G)$, the average degree
$$ d_{\rm ave}(G):= \frac{1}{\vert G\vert} \sum_{v\in V_G} d_G(v),$$
and the  degree-variance
\begin{equation} \label{var}
 d_{\rm var}(G):= \frac{1}{\vert G\vert} \sum_{v\in V_G} \left[d_G(v)-d_{\rm ave}(G) \right]^2
 \end{equation}
of the seed  $G$.  The degree-variance (\ref{var}) is a graph invariant proposed by Bell\,\cite{Bel} as measure of irregularity of a graph. Note that $d_{\rm var}(G)=0$ exactly when $G$ is regular.
Another  interesting graph invariant preserved in a self-similar sequence is the principal ratio
\begin{equation}\label{gamma}
 \gamma(G):= \frac{\max_{v\in V_G} x_G(v)}{\min_{v\in V_G} x_G(v)}\,.
 \end{equation}
  The  name of principal ratio is because
  (\ref{gamma}) corresponds to the ratio between the largest and the smallest component of the principal  eigenvector of $G$.
 The expression (\ref{gamma}) is well defined for any connected graph and can be used as an alternative  measure of irregularity, cf.\,Cioab\u{a} and Gregory\,\cite{CG}. We start by writing a lemma.
\begin{lemma} \label{le:useful} Let $G$ and $H$ be orbitally similar connected graphs.
Then:
\begin{itemize}
\item [(a)] $\delta(G) = \delta(H)$ and  $\Delta(G)= \Delta(H)$.
\item [(b)]  $d_{\rm ave}(G)= d_{\rm ave}(H)$ and  $d_{\rm var}(G)= d_{\rm var}(H)$.
\item [(c)]  $\gamma(G)=  \gamma(H)$.
\end{itemize}
\end{lemma}
\begin{proof} We  label the orbits $\mathcal O_1,\ldots,  \mathcal O_\ell$  of $G$ as in (\ref{cardi}) and, afterwards, we label the orbits  $\mathcal Q_1,\ldots,  \mathcal Q_\ell$ of $H$  so as to have
$S_G=S_H= [s_{i,j}]$. Let  $i\in \{1,\ldots, \ell\}$. The vertices in $\mathcal O_i$ are automorphically similar and therefore they have the same degree, namely,
$
d_G(\mathcal O_i):= \sum_{j=1}^\ell s_{i,j}.
$
 So, it makes sense to say that $d_G(\mathcal O_i)$ is the degree of the orbit $\mathcal O_i$. Since $S_G=S_H$, we have
\begin{equation}\label{dodq}
d_G(\mathcal O_i)= d_H(\mathcal Q_i) \qquad \mbox{ for all }i\in \{1,\ldots, \ell\}.
\end{equation}
This observation yields
\begin{eqnarray*}
\delta(G)= \min_{v\in V_G} d_G(u) =  \min_{1\leq i\leq \ell} d_G(\mathcal O_i)
= \min_{1\leq i\leq \ell} d_H(\mathcal Q_i)= \delta(H)
\end{eqnarray*}
and similarly for the  maximum degree. For proving  (b) it suffices to check that
$M_r(G)=M_r(H),$
where $$M_r(G):= \frac{1}{\vert G\vert} \sum_{v\in V_G} [d_G(v)]^r$$ is the non-centralized $r$-th moment of $G$.
For all positive integer $r$, we have
\begin{eqnarray*}
\vert G \vert \,M_r(G)= \sum_{v\in V_G} [d_G(v)]^r = \sum_{i=1}^\ell \sum_{v\in \mathcal O_i} [d_G(v)]^r= \sum_{i=1}^\ell \vert \mathcal O_i\vert  \,[d_G(\mathcal O_i)]^r.
\end{eqnarray*}
Hence,
\begin{equation*}
M_r(G)\,=\,  \sum_{i=1}^\ell \left(\frac{\vert \mathcal O_i\vert}{\vert G \vert}\right)  [d_G(\mathcal O_i)]^r \;= \;\sum_{i=1}^\ell \left(\frac{\vert \mathcal Q_i\vert}{\vert H \vert}\right)  [d_H(\mathcal Q_i)]^r\;=\; M_r(H),
\end{equation*}
where the second  equality is obtained by combining (\ref{claima}) and (\ref{dodq}).
Finally, we  examine the case of the principal ratio. Lemma\,\ref{le:nice} shows  that $\varrho(G)$ is the largest eigenvalue of $S:=[s_{i,j}]$. From the proof of Lemma\,\ref{le:nice}, we  know that the principal eigenvectors of $G$ and $H$ have  the form
\begin{eqnarray*}
x_G&=& (\overbrace{\alpha_1,\ldots,\alpha_1}^{n_1(G)},\,\overbrace{\alpha_2,\,\ldots,\alpha_2}^{n_2(G)},\, \ldots,\,\overbrace{\alpha_\ell,\ldots,\alpha_\ell}^{n_\ell(G)}\,)^\top,\\ [1.4mm]
x_H&=& (\underbrace{\beta_1,\ldots,\beta_1}_{n_1(H)},\,\underbrace{\beta_2,\ldots,\beta_2}_{n_2(H)}, \, \ldots,\,\underbrace{\beta_\ell,\ldots,\beta_\ell}_{n_\ell(H)}\,)^\top,
\end{eqnarray*}
and that
 $S\alpha =\varrho(G)\alpha$ and $S\beta =\varrho(G)\beta$
 with $\alpha:=(\alpha_1,\ldots,\alpha_\ell)^\top$ and $\beta:=(\beta_1,\ldots,\beta_\ell)^\top$. Beware that $x_G$ and $x_H$ do not have  necessarily   the same number of components. Anyway,  since  $S$ is an irreducible nonnegative matrix, we have   $\alpha= t\beta$ for some positive scalar $t>0$. Hence,
$$\gamma(G)=\frac{\max\{\alpha_1,\ldots,\alpha_\ell\}}{\min\{\alpha_1,\ldots,\alpha_\ell\}}
=\frac{\max\{t\beta_1,\ldots,t\beta_\ell\}}{\min\{t\beta_1,\ldots,t\beta_\ell\}}
=\frac{\max\{\beta_1,\ldots,\beta_\ell\}}{\min\{\beta_1,\ldots,\beta_\ell\}}=\gamma(H),$$
as desired.
\end{proof}

 In view  of  Lemma\,\ref{le:useful}, the next theorem is  evident and, therefore,  it it given without proof.

\begin{theorem} \label{th:other} Let $G$ be a connected graph and  $\{G_k:k\geq 1\}$ be  a self-similar sequence emanating from  $G$.  Then
$$\Phi(G_k)= \Phi(G)\qquad \mbox{for all } \,k\geq 1,$$
where $\Phi$ is any of the following graph invariants: minimum degree, maximum degree, average  degree, degree-variance, principal ratio.
\end{theorem}
\vskip 0,1cm
 The  celebrated degree sum formula implies that
\begin{equation}
\label{kappa}   \frac{ e_{G}}{\vert G\vert}\,=\, \frac{1}{2}\, d_{\rm ave}(G)\,,
\end{equation}
where $e_G$ is the  number of edges of $G$.
The edge-vertex ratio  (\ref{kappa}) is a graph invariant that measures the  ``density'' of $G$.  If such a  ratio is big, then the graph has a lot of edges compared to the number of vertices.  A more popular way of measuring the density  of a graph is by using the standard density index
\begin{equation}\label{dens}
{\rm dens}(G):= \frac{2 e_G}{\vert G\vert (\vert G\vert- 1)}\,.
\end{equation}
The expression  (\ref{dens})  is a fraction between $0$ and $1$ that tells us what portion of all possible edges are actually realized in the graph.  In contrast to the  edge-vertex ratio, the standard density index is not  preserved in a self-similar sequence. Indeed, 
the next proposition shows that $ {\rm dens}(\cdot)$ is a decreasing function
on any self-similar sequence.
\begin{proposition}\label{pr:density}
Let  $\{G_k: k\geq 1\}$ be  a self-similar sequence.  Then
${\rm dens}(G_k)$ decreases to $0$ as $k$ goes to infinity.
\end{proposition}
\begin{proof}
It suffices to write
$$ {\rm dens}(G_k)=  \left( \frac{ e_{G_k}}{\vert G_k\vert}\right)\left( \frac{2 }{ \vert G_k\vert- 1}\right) =  \left( \frac{ e_{G_1}}{\vert G_1\vert}\right)\left( \frac{2 }{ \vert G_k\vert- 1}\right)$$
and observe that $\vert G_k\vert$ grows to infinity with $k$.
\end{proof}

\subsection{Self-similar sequences and cyclicity}

A self-similar sequence may not preserve the cyclomatic number
$ c(G):= e_G-\vert G\vert +1 $
of the seed $G$. For instance, Figure\,\ref{Fig:biorbital}  displays the case of a self-similar sequence $\{G_k: k\geq 1\}$ of connected graphs with cyclomatic number $c(G_k)$ going to infinity with $k$. In general, whether the cyclomatic number  $c(G)$ of the seed $G$ is preserved or not depends on the seed itself. The next theorem fully clarifies this issue. There are three cases for consideration:
The case     $c(G)=0$ occurs when $G$ is a tree, the case $c(G)=1$ occurs when $G$ is unicyclic, and  the case   $c(G)\geq2 $ occurs when $G$ is bicyclic,  tricyclic, etcetera.

   \begin{theorem}\label{th:cyclo} Let $G$ be a connected graph.
   \begin{itemize}
   \item [(a)] If  $G$ is a tree,   then there is no self-similar sequence emanating from $G$.
   \item [(b)] If  $G$ is  a unicyclic graph and  $\{G_k: k\geq 1\}$ is a self-similar sequence  emanating from $G$, then each $G_k$ is a unicyclic graph.
          \item [(c)] If  $c(G)\geq 2$  and $\{G_k: k\geq 1\}$ is a self-similar sequence  emanating from $G$, then  $c(G_k)$ increases  to infinity as $k\to \infty$.
   \end{itemize}
   \end{theorem}
   \begin{proof}
  Let $G$ be connected graph with $n$ vertices and $m$ edges.
Let $\{G_k: k\geq 1\}$ be  a self similar sequence emanating from $G$.  Since $G_1$ is isomorphic to $G$, the term  $c_1:=c(G_1)$ is equal to $m-n+1$.  Pick any $k\geq 2$. The term  $c_k:=c(G_k)$ satisfies the relation
$$ \frac{c_k+ \vert G_k\vert -1}{\vert G_k\vert}=  \frac{c_1+ n -1}{n}$$
because  $G_k$ has the same edge-vertex ratio as $G_1$, cf.\,Theorem\,\ref{th:other}.  After simplification, we get
\begin{equation}\label{basic}
  c_k-1= (c_1-1) \,\frac{\vert G_k\vert}{n}\,.
  \end{equation}
 Several conclusions can be drawn from (\ref{basic}). For instance, if  $c_1=1$, then $c_k=1$. This observation takes care of part (b).
  Suppose now that  $c_1=0$. In such a case, the term on the right-hand side of (\ref{basic}) is negative. It follows that $c_k=0$ and, a posteriori, $\vert G_k\vert=n$. Since this equality contradicts  the growth condition (\ref{c1}), we deduce $c_1$ cannot be zero. This observation   settles (a). Finally, suppose that $c_1\geq 2$.  In such a case, $c_1-1 >0$ and the term on the right-hand side of (\ref{basic}) increases to infinity as $k\to \infty$. Hence, so does the cyclomatic number  of the graph $G_k$.
\end{proof}

\vskip 0,1cm

Hence, the cyclomatic number  of a seed is preserved by a self-similar sequence if and only if the seed is unicyclic. What we retain from Theorem\,\ref{th:cyclo} are the following three facts: firstly, a self-similar sequence  contains no tree. Secondly,
in a self-similar sequence, if one graph is unicyclic, then all graphs are unicyclic. And, thirdly, in a self-similar sequence, if one graph is not unicyclic, then  the successive cyclomatic numbers go up to infinity.

%





\begin{thebibliography}{10}
{\small

\bibitem{BFW}
W.\,Barrett, A.\,Francis, and  B.\,Webb.
Equitable decompositions of graphs with
symmetries. \emph{Linear Algebra Appl.},  513 (2017), 409--434.

\bibitem{Bel}
F.K.\,Bell. A note on the irregularity of graphs. \emph{Linear Algebra Appl.}, 161 (1992), 45-54.


\bibitem{CG}
S.M.\,Cioab\u{a} and D.A.\,Gregory.
Principal eigenvectors of irregular graphs. \emph{Elec. J. Linear Algebra}, 16 (2007), 366--379.



\bibitem{CRS}
D.\,Cvetkovi\'{c}, P.\,Rowlinson, and S.\,Simi\'{c}.
\emph{Eigenspaces of Graphs}. Cambridge Univ. Press, Cambridge, 1997.


\bibitem{DM}
M.\,Dehmer and A.\,Mowshowitz.
A history of graph entropy measures. \emph{Inform. Sci.}, 181 (2011) 57--78.

\bibitem{FH}
R.\, Frucht  and F.\,Harary.
On the corona of two graphs. \emph{Aequationes Math.}, 4 (1970), 322--325.



\bibitem{HIK}
R.\,Hammack, W.\,Imrich, and S.\,Klav\v{z}ar.
\emph{Handbook of Product Graphs}. CRC Press,  Taylor and  Francis Group, Florida, 2011.

\bibitem{HP}
F.\,Harary and E.\,Palmer.
On  similar  points of a graph.
\emph{J. Math. Mech.}, 15 (1966), 623--630.





\bibitem{Mow}
A.\,Mowshowitz.
Entropy and the complexity of graphs: I. An index of the relative
complexity of a graph. \emph{Bull. Math. Biophys.},  30 (1968), 175--204.

\bibitem{MM}
A.\,Mowshowitz and V.\,Mitsou.
Entropy, orbits, and spectra of graphs.
In: \emph{Analysis of Complex Networks: From Biology to Linguistics}, M.\,Dehmer, ed., Wiley-VCH, Weinheim (2009), 1--22.

\bibitem{SeSo1}
A.\,Seeger and D.\,Sossa.
Vertex-removal,  vertex-addition,  and different notions of similarity for vertices of a graph.  \emph{Linear Multilinear Algebra},  70 (2022), 5173--5192.


}
\end{thebibliography}
\end{document}